\documentclass{amsart}

\usepackage{amsmath,amssymb,mathtools}
\usepackage{enumitem}
\usepackage{booktabs}
\usepackage{xcolor}
\usepackage[hypertexnames=false]{hyperref}
\hypersetup{colorlinks=true,linkcolor=blue,citecolor=blue,urlcolor=blue}

\newcommand{\Z}{\mathbb Z}
\newcommand{\Q}{\mathbb Q}
\newcommand{\R}{\mathbb R}
\newcommand{\HH}{\mathbb H}
\newcommand{\OO}{\mathcal O}
\newcommand{\Norm}{\operatorname N}

\newcommand{\ECal}{\operatorname{ECal}}
\newcommand{\lex}{\lambda_{\mathrm{ex}}}
\newcommand{\Xred}{\mathcal X}
\newcommand{\Xex}{\mathcal X^{\mathrm{ex}}}
\newcommand{\Xone}{\mathcal X_{1}}
\newcommand{\dfn}{\mathfrak d}
\newcommand{\vol}{\operatorname{vol}}
\newcommand{\PGL}{\mathrm{PGL}}
\newcommand{\PSL}{\mathrm{PSL}}
\newcommand{\Lv}{\mathsf m}
\DeclareMathOperator{\GL}{GL}

\newtheorem{theorem}{Theorem}[section]
\newtheorem{proposition}[theorem]{Proposition}
\newtheorem{lemma}[theorem]{Lemma}
\newtheorem{corollary}[theorem]{Corollary}
\theoremstyle{definition}

\theoremstyle{remark}

\numberwithin{equation}{section}

\title[Reduced ideals missed by the reduction algorithm]
{Density and Defect Distribution of Reduced Ideals Missed by the Reduction Algorithm in Real Quadratic Fields}

\author{Mengzhe Guo}
\address{Graduate School of Science and Engineering, Ritsumeikan University, 1-1-1 Nojihigashi, Kusatsu, Shiga 525-8577, Japan}
\email{guozhechan@gmail.com}

\author{Taro Hayashi}
\address{Department of Mathematical Sciences, College of Science and Engineering, Ritsumeikan University, 1-1-1 Nojihigashi, Kusatsu, Shiga 525-8577, Japan}
\email{haya4taro@gmail.com}
\thanks{T. Hayashi is the corresponding author. ORCID: 0009-0005-0145-8042.}

\author{Ryota Sueyoshi}
\address{Graduate School of Science and Engineering, Ritsumeikan University, 1-1-1 Nojihigashi, Kusatsu, Shiga 525-8577, Japan}
\email{ryota.sueyoshi@outlook.jp}


\begin{document}

\begin{abstract}
Schoof asked how many reduced Arakelov divisors
cannot arise from the reduction algorithm.
Tian and Tran subsequently characterized
the image of this algorithm in real quadratic
fields as precisely the $1$-reduced ideals.
We give a quantitative answer to Schoof's
question in the real-quadratic setting by
determining the asymptotic proportion of
reduced ideals missed by the algorithm.
A key arithmetic input is the criterion
that a reduced ideal with associated reduced
form $(A,b,-r)$ fails to be $1$-reduced
if and only if $A-r<b<2(A-r)$.
Hence the missed ideals occupy an explicit
region of the natural extension of the
Gauss map whose relative measure is
$\lex=\frac{\log(4/3)}{\log4}$.
In particular, the missed region occupies
exactly half of the digit-one cylinder
with respect to the invariant measure.
We establish this limiting proportion
in two settings.
A lattice-point count gives the average
asymptotic over fundamental discriminants
$\Delta\le X$, with error
$O(X^{-1/4+\varepsilon})$.
For individual real quadratic fields,
Duke's equidistribution theorem yields
$\ECal(\Delta)/\kappa(\Delta)\to\lex$
as $\Delta\to\infty$.
We also determine the limiting distribution
of the canonical defect $k/\sqrt\Delta$
of the missed ideals.
This distribution is supported on
$[0,1/\sqrt3-1/2]$, and its lower tail
is asymptotic to
$\frac{2}{\log(4/3)}t\log(1/t)$
as $t\to0^+$.
An exact relation between the integer defect
and an endpoint coordinate of the natural
extension allows us to transfer
equidistribution to the defect law.
\end{abstract}

\maketitle

\section{Introduction}\label{sec:intro}

Reduced ideals in real quadratic fields connect
continued fractions, binary quadratic forms,
ideal classes, and Shanks's infrastructure;
see \cite{Buell1989,HalterKoch2013,Schoof2008,Shanks1972}.
In his analysis of algorithms for Arakelov class
groups, Schoof observed that certain reduced
Arakelov divisors cannot occur as outputs of
the reduction algorithm.
He explicitly asked
\emph{``It would be of interest to know how many
such reduced divisors there may be.''}
\cite[p.~492]{Schoof2008}.
Tian and Tran took up this problem and described
their results as a partial answer to Schoof's
question \cite{TianTran2024}.
In real quadratic fields, they characterized
the possible outputs of the reduction algorithm
as exactly the $1$-reduced ideals.
For example, $\Q(\sqrt{73})$ has nine reduced
ideals, but only seven are $1$-reduced.
The present paper gives a quantitative answer
to Schoof's question in the real-quadratic
setting.
More precisely, we ask how many reduced ideals
are inaccessible to the algorithm as the
fundamental discriminant grows, and what
proportion of all reduced ideals they represent.
We answer this both on average over
discriminants and for individual real quadratic
fields.

For a positive fundamental discriminant
$\Delta$, let $\kappa(\Delta)$ be the number
of reduced ideals of $\Q(\sqrt\Delta)$
(the caliber, cf.\ \cite{JunLee2010}),
and let $\ECal(\Delta)$ be the number
of reduced ideals that are not $1$-reduced.
Throughout, $\sum^\flat$ denotes a sum
over positive fundamental discriminants, and
\[
 \lex=\frac{\log(4/3)}{\log4}.
\]

\begin{theorem}\label{thm:intro-average}
Put
$\mathfrak S
 =\frac{11}{16}
 \prod_{p>2}
 \bigl(1-p^{-2}-p^{-3}+p^{-4}\bigr)$.
For every $\varepsilon>0$, we have
\begin{align*}
 \sideset{}{^\flat}\sum_{\Delta\le X}
 \kappa(\Delta)
 &=
 \frac{\mathfrak S\log2}{3}X^{3/2}
 +O_\varepsilon(X^{5/4+\varepsilon}),\\
 \sideset{}{^\flat}\sum_{\Delta\le X}
 \ECal(\Delta)
 &=
 \frac{\mathfrak S\log(4/3)}{6}X^{3/2}
 +O_\varepsilon(X^{5/4+\varepsilon}).
\end{align*}
In particular, the proportion of reduced
ideals missed by the reduction algorithm,
averaged over $\Delta\le X$, is
$\lex+O_\varepsilon(X^{-1/4+\varepsilon})$.
\end{theorem}

\begin{theorem}\label{thm:intro-individual}
As $\Delta\to\infty$ through positive
fundamental discriminants, we have
\[
 \frac{\ECal(\Delta)}{\kappa(\Delta)}
 \longrightarrow\lex.
\]
Moreover, among the reduced ideals
$\Z+\Z f$ with $\lfloor f\rfloor=1$,
the proportion of those that are not
$1$-reduced tends to $1/2$.
In particular,
$\ECal(\Delta)\gg_\varepsilon
\Delta^{1/2-\varepsilon}$
\textup{(}ineffectively\textup{)}.
\end{theorem}

The two theorems quantify the missed set
in complementary ways.
Theorem~\ref{thm:intro-average} counts
reduced and missed ideals in aggregate
over fundamental discriminants, whereas
Theorem~\ref{thm:intro-individual} gives
the limiting ratio for each field as
its discriminant grows.
We emphasize that these results concern
which reduced ideals \emph{can} occur
as outputs of the reduction algorithm.
They do not assert that actual algorithm
outputs are uniformly sampled, nor do
they imply a new complexity bound for
class-group algorithms.
Both theorems rest on an elementary
observation (Lemma~\ref{lem:linear-coordinates}).
Write a reduced ideal as $\Z+\Z f$ with
\[
 f=\frac{b+\sqrt\Delta}{2A},
 \qquad
 r=\frac{\Delta-b^2}{4A},
\]
so that $(A,b,-r)$ is the associated
reduced form.
Then $\Z+\Z f$ is reduced if and only if
$|A-r|<b$, and it fails to be
$1$-reduced if and only if
\[
 A-r<b<2(A-r),
 \qquad\text{equivalently}\qquad
 (2f-1)(2\bar f-1)>-3.
\]
Thus the non-$1$-reduced locus is cut out
by \emph{linear} inequalities in the
coordinates of the form, and in the
endpoint coordinates $(f,\bar f)$ it is
an explicit region $\Xex$ inside the
reduced region
\[
 \Xred=\{f>1,\ -1<\bar f<0\}.
\]
The pairs $(f,\bar f)$ are governed
by the measure
\[
 d\nu=\frac{df\,d\bar f}{(f-\bar f)^2}
\]
on $\Xred$, the invariant measure
of the natural extension of the Gauss map.

On average over $\Delta$, this is a
lattice-point count in a cone
(Section~\ref{sec:average}).
For an individual field, it follows
from Duke's theorem \cite{Duke1988},
transported to a cross-section of the
geodesic flow
(Section~\ref{sec:equidistribution}).
In both cases, the limiting proportion
is $\nu(\Xex)/\nu(\Xred)$.
An elementary integration gives
\[
 \nu(\Xred)=\log2,
 \qquad
 \nu(\Xex)=\frac12\log\frac43
\]
(Lemma~\ref{lem:measures}).
The digit-one region, defined by
$\lfloor f\rfloor=1$, has
$\nu$-measure $\log(4/3)$.
Since the missed region is contained
in this region, it occupies exactly
half of its measure.

We next use the same endpoint coordinates
to describe the limiting distribution
of the normalized canonical defect
$k(I)/\sqrt\Delta$ among the missed ideals.
To state this distribution, put
\[
 t_{\max}=\frac1{\sqrt3}-\frac12.
\]
For $0<t<t_{\max}$, define
\begin{align*}
 u_\pm(t)
 &=\frac{2t+3\pm\sqrt{1-12t-12t^2}}
 {2(2t+1)},\\
 F_t(u)
 &=(1+t)\log u-t\log(u-1)
 -\frac12\log(u^2-u+1).
\end{align*}
We then define
\[
 G(t)=
 \begin{cases}
 \frac12\log\frac43,
   & t=0,\\
 F_t(u_+(t))-F_t(u_-(t)),
   & 0<t<t_{\max},\\
 0, & t\ge t_{\max}.
 \end{cases}
\]

\begin{theorem}
\label{thm:intro-defect}
As $\Delta\to\infty$ through positive
fundamental discriminants, the normalized
defects of the missed reduced ideals
have the limiting distribution
\[
 \frac{
 \#\{I\text{ reduced, not $1$-reduced}:
 k(I)\le t\sqrt\Delta\}
 }{\ECal(\Delta)}
 \longrightarrow
 1-\frac{G(t)}{G(0)}
\]
for every $t\ge0$.
The limiting distribution is supported
on $[0,t_{\max}]$.
Moreover, as $t\to0^+$,
\[
 1-\frac{G(t)}{G(0)}
 =
 \frac{2}{\log(4/3)}
 \left(
 t\log\frac1t
 +(1-2\log2)t
 \right)
 +o(t).
\]
\end{theorem}

The support endpoint is the maximum
of the defect function on the closure
of the exceptional region;
see Lemma~\ref{lem:measures}.

\section{Defect coordinates and the exceptional region}
\label{sec:prelim}

Let $K=\Q(\sqrt\Delta)$, where $\Delta$
is a positive fundamental discriminant.
Following the normalization of
Tian--Tran \cite{TianTran2024},
every reduced fractional ideal has
a unique presentation of the form
\[
 I=I(A,b)=\Z+\Z f,\qquad
 f=\frac{b+\sqrt\Delta}{2A},
\]
where $A>0$ and
\begin{equation}\label{eq:reduced-interval}
 |\sqrt\Delta-2A|<b<\sqrt\Delta.
\end{equation}
Moreover, $\Norm(I^{-1})=A$.
Conversely, if
$b^2\equiv\Delta\pmod{4A}$ and
\eqref{eq:reduced-interval} holds,
the displayed module is a reduced
fractional ideal by the standard
quadratic-ideal construction.
In these coordinates, Tian--Tran's
criterion takes the form
\begin{equation}\label{eq:TT-compact}
 I(A,b)\text{ is not }1\text{-reduced}
 \quad\Longleftrightarrow\quad
 (b-A)^2+3A^2>\Delta.
\end{equation}
Indeed, their condition
$\Norm(f-1/2)>-3/4$ is equivalent
to \eqref{eq:TT-compact}.

Let $T$ be the least integer strictly
larger than $\sqrt\Delta$ such that
$T\equiv\Delta\pmod2$, and put
\[
 U=\frac{T^2-\Delta}{4}.
\]
By construction,
$T-2<\sqrt\Delta<T$, which implies
$U\ge1$ and $T>U+1$.

\begin{lemma}\label{lem:general-defect}
For a reduced ideal $I(A,b)$, put
\[
 d=\frac{T-b}{2},\qquad e=\frac{T-2A+b}{2}.
\]
Then $d,e\in\Z_{>0}$, $d+e=T-A$, and
\[
 k(I):=\frac{de-U}{A}\in\Z_{\ge0}.
\]
We call $k(I)$ the \emph{defect} of $I$.
\end{lemma}
\begin{proof}
Since
$b^2\equiv\Delta\equiv T^2\pmod4$,
we have $b\equiv T\pmod2$.
Thus $d$ and $e$ are integers.
The reducedness condition
\eqref{eq:reduced-interval} gives
$b<\sqrt\Delta<T$ and $b>2A-\sqrt\Delta$.
These inequalities imply $d,e>0$.
A direct calculation gives
$d+e=T-A$ and $de-U=\frac{\Delta-b^2}{4}-Ad$.
Since
$b^2\equiv\Delta\pmod{4A}$,
we obtain $A\mid de-U$.
Since $d,e\ge1$, we have
$de\ge d+e-1=T-A-1$.
As $T>U+1$, it follows that
$de-U\ge T-A-1-U>-A$.
Since $de-U$ is divisible by $A$,
we conclude that $de-U\ge0$.
Hence $k(I)\in\Z_{\ge0}$.
\end{proof}

\subsection{Endpoint coordinates and the natural-extension measure}\label{subsec:endpoint}
For a reduced ideal $I=I(A,b)$, put
\[
 r=r(I)=\frac{\Delta-b^2}{4A}.
\]
Then
$Q_I(x,y)=Ax^2+bxy-ry^2$
is a binary quadratic form of
discriminant $\Delta$.
Let
\[
 \Xred=\{(w_+,w_-)\in\R^2:\ w_+>1,\ -1<w_-<0\},
 \qquad
 d\nu=\frac{dw_+\,dw_-}{(w_+-w_-)^2},
\]
and define the subsets
\begin{align*}
 \Xone&=\{(w_+,w_-)\in\Xred:\ w_+<2\},\\
 \Xex&=\{(w_+,w_-)\in\Xred:\ (2w_+-1)(2w_--1)>-3\},
\end{align*}
and the function
\[
 \dfn(w_+,w_-)=\frac{(w_+-1)(-w_-)}{w_+-w_-}.
 \]
To a reduced ideal $I=\Z+\Z f$ of $K$, we attach the point $\xi_I=(f,\bar f)$.
We write $\Xred_\Delta=\{\xi_I : I \text{ is a reduced ideal of } K\}$ for the set of these points.

\begin{lemma}\label{lem:linear-coordinates}
Let $\Delta$ be a positive fundamental
discriminant.
\begin{enumerate}[label=\textup{(\roman*)}]
\item
The map $I(A,b)\mapsto(A,b,r)$
is a bijection from the set of
reduced ideals of $K$ onto
$\{(A,b,r)\in\Z_{>0}^3:
 b^2+4Ar=\Delta,\ |A-r|<b\}$.

\item
For $I=I(A,b)$ and
$f=(b+\sqrt\Delta)/(2A)$, we have
\[
 f+\bar f=\frac bA,\qquad
 f\bar f=-\frac rA,\qquad
 f-\bar f=\frac{\sqrt\Delta}{A}.
\]
Moreover, $\xi_I\in\Xred$, and
the map $I\mapsto\xi_I$ is injective.

\item
The ideal $I$ is not $1$-reduced
if and only if
\[A-r<b<2(A-r),\]
or equivalently, $\xi_I\in\Xex$.
Moreover, $\Xex\subset\Xone$.
Thus every non-$1$-reduced ideal
satisfies $\lfloor f\rfloor=1$.

\item
The defect of $I$ satisfies
\[k(I)=\sqrt\Delta\,\dfn(\xi_I)
 -\frac{T-\sqrt\Delta}{2},\]
where
$0<\frac{T-\sqrt\Delta}{2}<1$.
\end{enumerate}
\end{lemma}

\begin{proof}
(i) For integers $A,b>0$, the conditions
$b<\sqrt\Delta$ and
$b^2\equiv\Delta\pmod{4A}$ are equivalent
to $r$ being a positive integer.
For $A,b,r>0$ with $b^2+4Ar=\Delta$,
we have
\begin{align*}
 |\sqrt\Delta-2A|<b
 &\iff \Delta-b^2+4A^2<4A\sqrt\Delta\\
 &\iff A+r<\sqrt\Delta\\
 &\iff (A+r)^2<b^2+4Ar\\
 &\iff |A-r|<b.
\end{align*}
Since $\Delta$ is fundamental, every
integral binary quadratic form of
discriminant $\Delta$ is primitive.
By the standard quadratic-ideal
construction recalled after
\eqref{eq:reduced-interval}, each such
triple therefore determines a reduced
fractional ideal.
The normalized presentation is unique.

(ii) The three identities follow
directly from the definition of $f$.
Moreover, the conditions
$f>1$ and $-1<\bar f<0$ are equivalent
to
\[
 b>2A-\sqrt\Delta,\qquad
 b>\sqrt\Delta-2A,\qquad
 b<\sqrt\Delta,
\]
which together give
\eqref{eq:reduced-interval}.
Since $f$ uniquely determines $(A,b)$,
the map $I\mapsto\xi_I$ is injective.

(iii) Using $\Delta=b^2+4Ar$, we obtain
\[
 (b-A)^2+3A^2-\Delta
 =2A(2A-b-2r).
\]
Hence, by \eqref{eq:TT-compact},
$I$ is not $1$-reduced if and only if
$b<2(A-r)$.
Together with $b>|A-r|$ from (i),
this is equivalent to
\[
 A-r<b<2(A-r).
\]
Dividing $2A-b-2r>0$ by $A$
and using (ii), we obtain
$2+2f\bar f-(f+\bar f)>0$,
which is equivalent to
\[(2f-1)(2\bar f-1)>-3.\]
If $f\ge2$, then
$2f-1\ge3$ and $2\bar f-1<-1$.
This contradicts the defining inequality
for $\Xex$.
Hence $f<2$ whenever $I$ is not
$1$-reduced.

(iv) By the definition of the defect
and $\Delta=T^2-4U$, we have
\[
\begin{aligned}
 k(I)
 &=\frac{de-U}{A}\\
 &=\frac{\Delta-b^2}{4A}
   -\frac{T-b}{2}\\
 &=r-\frac{T-b}{2}.
\end{aligned}
\]
On the other hand, by (ii),
\begin{align*}
 \sqrt\Delta\,\dfn(f,\bar f)
 &=A(f-1)(-\bar f)\\
 &=-Af\bar f+A\bar f\\
 &=r+\frac{b-\sqrt\Delta}{2}.
\end{align*}
Therefore,
\[
\begin{aligned}
 \sqrt\Delta\,\dfn(f,\bar f)
 -\frac{T-\sqrt\Delta}{2}
 &=r-\frac{T-b}{2}\\
 &=k(I).
\end{aligned}
\]
By the choice of $T$, we have
$0<T-\sqrt\Delta<2$, which implies
$0<\frac{T-\sqrt\Delta}{2}<1$.
\end{proof}

For $1<u<2$, define
$\varphi(u)=\frac12-\frac{3}{2(2u-1)}$.
\begin{lemma}\label{lem:measures}
The exceptional region is given by
\[
 \Xex=
 \{(w_+,w_-)\in\Xred
 \mid 1<w_+<2,\ 
 \varphi(w_+)<w_-<0\}.
\]
The measures of the three regions are
\[
 \nu(\Xred)=\log2,\qquad
 \nu(\Xone)=\log\frac43,\qquad
 \nu(\Xex)=\frac12\log\frac43.
\]
In particular,
\[
 \frac{\nu(\Xex)}{\nu(\Xred)}=\lex,
 \qquad
 \frac{\nu(\Xex)}{\nu(\Xone)}=\frac12.
\]
The defect function also satisfies
\[
 \sup_{\Xex}\dfn
 =\frac1{\sqrt3}-\frac12
 =t_{\max}.
\]
\end{lemma}

\begin{proof}
Write $u=w_+$ and $v=w_-$.
If $u\ge2$ and $-1<v<0$, then
$(2u-1)(2v-1)<-3$.
Hence $\Xex\subset\Xone$.
For $1<u<2$, the defining inequality
of $\Xex$ is equivalent to
\[
 v>\frac12-\frac{3}{2(2u-1)}
 =\varphi(u).
\]
Since $-1<\varphi(u)<0$, we obtain
$\Xex=
 \{(u,v):1<u<2,\ \varphi(u)<v<0\}$.

For fixed $u>1$, we have
\[
 \int_{v_0}^{0}\frac{dv}{(u-v)^2}
 =\frac1u-\frac1{u-v_0}.
\]
Taking $v_0=-1$ gives
\[
\begin{aligned}
 \nu(\Xred)
 &=\int_1^\infty
 \left(\frac1u-\frac1{u+1}\right)du
 =\log2,\\
 \nu(\Xone)
 &=\int_1^2
 \left(\frac1u-\frac1{u+1}\right)du
 =\log\frac43.
\end{aligned}
\]
Since
$u-\varphi(u)
 =\frac{2(u^2-u+1)}{2u-1}$,
we obtain
\begin{align*}
 \nu(\Xex)
 &=\int_1^2\int_{\varphi(u)}^0
   \frac{dv\,du}{(u-v)^2}\\
 &=\int_1^2
   \left(
   \frac1u-\frac{2u-1}{2(u^2-u+1)}
   \right)du\\
 &=\left[
   \log u-\frac12\log(u^2-u+1)
   \right]_1^2\\
 &=\frac12\log\frac43.
\end{align*}
The stated ratios follow immediately.

It remains to determine the supremum
of the defect function.
For fixed $1<u<2$, we have
\[
 \frac{\partial}{\partial v}\dfn(u,v)
 =-\frac{u(u-1)}{(u-v)^2}<0.
\]
Hence $\dfn(u,v)$ is strictly
decreasing in $v$, and its supremum
is approached as $v\to\varphi(u)^+$.
We therefore define
\[
 h(u)=\dfn(u,\varphi(u))
 =\frac{(u-1)(2-u)}
 {2(u^2-u+1)}.
\]
It follows that
$\sup_{\Xex}\dfn
 =\sup_{1<u<2}h(u)$.
Differentiating, we obtain
\[
 h'(u)=
 \frac{-2u^2+2u+1}
 {2(u^2-u+1)^2}.
\]
It follows that $h$ attains its
maximum on $(1,2)$ at
$u=(1+\sqrt3)/2$, where
$h\left(\frac{1+\sqrt3}{2}\right)
 =\frac1{\sqrt3}-\frac12$.
Therefore,
$\sup_{\Xex}\dfn
 =\frac1{\sqrt3}-\frac12
 =t_{\max}$.
\end{proof}

\section{The average proportion}\label{sec:average}
In this section, we prove
Theorem~\ref{thm:intro-average}.
We set $q(A,b,r)=b^2+4Ar$.
For a Borel set $B\subset\Xred$, put
\[
 \mathcal K_B=
 \Bigl\{
 (A,b,r)\in\R_{>0}^3:
 q(A,b,r)\le1,\
 \Bigl(
 \frac{b+\sqrt q}{2A},
 \frac{b-\sqrt q}{2A}
 \Bigr)\in B
 \Bigr\}.
\]
By Lemma~\ref{lem:linear-coordinates},
\begin{align*}
 \mathcal K_{\Xred}
 &=\{(A,b,r)\in\R_{>0}^3:
 q\le1,\ |A-r|<b\},\\
 \mathcal K_{\Xex}
 &=\{(A,b,r)\in\R_{>0}^3:
 q\le1,\ A-r<b<2(A-r)\}.
\end{align*}
Both regions are bounded, since
$A+r<\sqrt q\le1$ and $b\le1$.
They are defined by finitely many
polynomial inequalities of degree
at most two.

We now relate the reduced ideals
to lattice points in these regions.
The polynomial $q$ is homogeneous
of degree two, while the endpoint
coordinates
\[
 \frac{b\pm\sqrt{q(A,b,r)}}{2A}
\]
are invariant under positive scaling
of $(A,b,r)$.
Consequently, for $X>0$, a point
$(A,b,r)\in\R_{>0}^3$ belongs to
$X^{1/2}\mathcal K_B$ if and only if
\[
 q(A,b,r)\le X
\]
and its endpoint coordinates belong
to $B$.

Suppose now that $(A,b,r)$ is an
integer point with
$q(A,b,r)=\Delta$ a positive
fundamental discriminant.
Then
\[
 b^2\equiv\Delta\pmod{4A}.
\]
By Lemma~\ref{lem:linear-coordinates},
the condition that its endpoint
coordinates belong to $B\subset\Xred$
is equivalent to the associated ideal
$I(A,b)$ being reduced and satisfying
$\xi_I\in B$.
Conversely, every such reduced ideal
determines a unique integer triple
$(A,b,r)$.

Thus, for $B\in\{\Xred,\Xex\}$,
we obtain
\begin{equation}\label{eq:sum-as-lattice-count}
\begin{aligned}
 &\sideset{}{^\flat}\sum_{\Delta\le X}
 \#\bigl\{
 I\text{ a reduced fractional ideal of }
 \Q(\sqrt\Delta):
 \xi_I\in B
 \bigr\}\\
 &\qquad=
 \#\bigl\{
 x\in\Z^3\cap X^{1/2}\mathcal K_B:
 q(x)\text{ is a fundamental discriminant}
 \bigr\}.
\end{aligned}
\end{equation}
For $B=\Xred$, the left-hand side
equals
$\sideset{}{^\flat}\sum_{\Delta\le X}
\kappa(\Delta)$.
For $B=\Xex$, it equals
$\sideset{}{^\flat}\sum_{\Delta\le X}
\ECal(\Delta)$.

We first compute the volumes of
the corresponding regions.

\begin{lemma}\label{lem:volume}
For every Borel set $B\subset\Xred$, $\vol(\mathcal K_B)=\nu(B)/3$.  In particular $\vol(\mathcal K_{\Xred})=\frac{\log2}{3}$ and $\vol(\mathcal K_{\Xex})=\frac{\log(4/3)}{6}$.
\end{lemma}

\begin{proof}
We make the change of coordinates from $(A,b,r)$ to $(Q,w_+,w_-)$ by setting
\[Q=q(A,b,r)=b^2+4Ar,\qquad w_\pm=\frac{b\pm\sqrt Q}{2A}.\]
Thus a point $(A,b,r)\in\mathcal K_B$ is represented in the new
coordinates by $(Q,w_+,w_-)$, where $0<Q\le1$ and
$(w_+,w_-)\in B$.
We compute the Jacobian by factoring the change of variables as
\[
(Q,w_+,w_-)\longmapsto (Q,A,b)\longmapsto (A,b,r).
\]
By the chain rule for Jacobians, the corresponding determinants multiply.
For the second map, keeping $(A,b)$ fixed, the relation
$Q=b^2+4Ar$ gives
\[\left|
 \frac{\partial(A,b,r)}{\partial(A,b,Q)}
 \right|
 =
 \frac{\partial r}{\partial Q}
 =
 \frac{1}{4A}.\]
For the first map, keeping $Q$ fixed, the inverse transformation is
\[A=\sqrt Q/(w_+-w_-),\qquad b=\sqrt Q\,(w_++w_-)/(w_+-w_-).\]
Differentiating with respect to
$w_+$ and $w_-$ with $Q$ fixed,
we obtain
\begin{align*}
 \frac{\partial A}{\partial w_+}
 &=-\frac{\sqrt Q}{(w_+-w_-)^2},
&
 \frac{\partial A}{\partial w_-}
 &=\frac{\sqrt Q}{(w_+-w_-)^2},\\
 \frac{\partial b}{\partial w_+}
 &=-\frac{2\sqrt Q\,w_-}{(w_+-w_-)^2},
&
 \frac{\partial b}{\partial w_-}
 &=\frac{2\sqrt Q\,w_+}{(w_+-w_-)^2}.
\end{align*}
Hence
\begin{align*}
 \left|
 \frac{\partial(Q,A,b)}
 {\partial(Q,w_+,w_-)}
 \right|
 &=
 \left|
 \frac{\partial(A,b)}
 {\partial(w_+,w_-)}
 \right|\\
 &=
 \frac{Q}{(w_+-w_-)^4}
 \left|
 \begin{vmatrix}
 -1 & 1\\
 -2w_- & 2w_+
 \end{vmatrix}
 \right|\\
 &=
 \frac{2Q}{(w_+-w_-)^3}.
\end{align*}

Therefore, by the chain rule and
$1/(4A)=(w_+-w_-)/(4\sqrt Q)$,
\[
 dA\,db\,dr
 =
 \frac{1}{4A}
 \frac{2Q}{(w_+-w_-)^3}
 \,dQ\,dw_+\,dw_-
 =
 \frac{\sqrt Q}{2}\,dQ\,d\nu(w_+,w_-).
\]
Therefore the integration separates into the $Q$-coordinate and the
endpoint coordinates.  Integrating over $0<Q\le1$ and
$(w_+,w_-)\in B$ gives
$\vol(\mathcal K_B)=\frac13\,\nu(B)$.
The stated values of $\vol(\mathcal K_{\Xred})$ and
$\vol(\mathcal K_{\Xex})$ now follow from
Lemma~\ref{lem:measures}.
\end{proof}

Let 
\[\mathcal F_{16}=\{1,5,8,9,12,13\}\subset\Z/16\Z.\]
A positive integer $\Delta\equiv0,1\pmod4$ is a fundamental discriminant if and only if $\Delta\bmod16\in\mathcal F_{16}$ and $p^2\nmid\Delta$ for every odd prime $p$.
\begin{lemma}\label{lem:local-densities}
Write $x=(A,b,r)$.
Then
\[
 \varrho_{16}:=\#\{x\in(\Z/16\Z)^3:\ q(x)\bmod16\in\mathcal F_{16}\}=2816.
\]
For an odd prime $p$, we have
\[
 \varrho(p^2):=\#\{x\in(\Z/p^2\Z)^3:\ q(x)\equiv0\pmod{p^2}\}=p^4+p^3-p^2 .
\]
\end{lemma}

\begin{proof}
We first compute $\varrho_{16}$.
If $b$ is odd, then
$q(x)\equiv1\pmod4$, so
\[
 q(x)\bmod16\in\{1,5,9,13\}
 \subset\mathcal F_{16}
\]
for all $A,r$.
These triples contribute
$8\cdot16^2=2048$ solutions.
Suppose that $b$ is even and write
$b=2c$.
Then
$q(x)=4(c^2+Ar)$,
and the condition
$q(x)\bmod16\in\mathcal F_{16}$
is equivalent to
\[
 c^2+Ar\equiv2,3\pmod4.
\]
Among the $16$ pairs $(A,r)\bmod4$,
the product $Ar$ takes the values
$0,1,2,3$ exactly $8,2,4,2$ times,
respectively.
If $c$ is even, the condition is
$Ar\equiv2,3\pmod4$.
If $c$ is odd, it is
$Ar\equiv1,2\pmod4$.
In either case, exactly $6$ of the
$16$ pairs qualify.
Thus a proportion $3/8$ of the
triples with $b$ even satisfy the
required condition.
Their contribution is
$\frac38\cdot8\cdot16^2=768$.
Consequently,
$\varrho_{16}=2048+768=2816$.

Now let $p$ be an odd prime.
We first count the solutions of
$q(A,b,r)=0$ over $\mathcal F_p$.
If $A\neq0$, each choice of $A$ and $b$
uniquely determines $r$, giving
$p(p-1)$ solutions.
If $A=0$, then $b=0$, while $r$ is
arbitrary, giving $p$ solutions.
Thus there are exactly $p^2$ solutions
over $\mathcal F_p$.
The gradient of $q$ is
$\nabla q=(4r,2b,4A)$.
Since $p$ is odd, its only zero
over $\mathcal F_p$ is the origin.
Hence all $p^2-1$ nonzero solutions
of $q=0$ modulo $p$ are nonsingular.
Fix one such solution $x_0$ modulo $p$
and choose an integral representative.
Every lift of $x_0$ modulo $p^2$
has the form $x_0+py$, where
$y\in(\Z/p\Z)^3$.
We have
\[
 q(x_0+py)
 \equiv q(x_0)+p\nabla q(x_0)\cdot y
 \pmod{p^2}.
\]
Since $\nabla q(x_0)\neq0\pmod p$,
the condition
$q(x_0+py)\equiv0\pmod{p^2}$
gives one nontrivial affine linear
equation in the three coordinates
of $y$.
This equation has exactly $p^2$
solutions.
Thus the nonzero solutions modulo $p$
contribute $(p^2-1)p^2$ solutions
modulo $p^2$.

It remains to count the solutions
with $x\equiv0\pmod p$.
Writing $x=py$ with
$y\in(\Z/p\Z)^3$, the homogeneity
of $q$ gives
$q(x)=p^2q(y)$.
Hence all $p^3$ such residues
are solutions modulo $p^2$.

Combining these contributions, we obtain
\[
\begin{aligned}
 \varrho(p^2)
 &=(p^2-1)p^2+p^3\\
 &=p^4+p^3-p^2.
\end{aligned}
\]

\end{proof}

\begin{lemma}\label{lem:crude}
For $m\ge1$ let 
\[\kappa^*(m)=\#\{(A,b,r)\in\Z_{>0}^3:\ q(A,b,r)=m,\ |A-r|<b\}.\]
Then for every $\varepsilon>0$,
\[
 \kappa^*(m)\ll_\varepsilon m^{1/2+\varepsilon}.
\]
\end{lemma}

\begin{proof}
Since $q(A,b,r)=b^2+4Ar=m$ and $A,r>0$, we have
$b<\sqrt m$, so there are $O(\sqrt m)$ possible values of $b$.
For each fixed $b$ with $b^2\equiv m\pmod 4$, the relation
$Ar=\frac{m-b^2}{4}$
shows that the pair $(A,r)$ is determined by a positive divisor $A$
of $(m-b^2)/4$.  The additional condition $|A-r|<b$ can only
decrease the number of such pairs.  Hence, by the standard divisor
bound $\tau(n)\ll_\varepsilon n^\varepsilon$
\cite[Theorem~315]{HardyWright},
we have that 
\[\kappa^*(m)
 \le
 \sum_{\substack{b<\sqrt m\\ b^2\equiv m\pmod 4}}
 \tau\!\left(\frac{m-b^2}{4}\right)
 \ll_\varepsilon
 \sqrt m\,m^\varepsilon
 =
 m^{1/2+\varepsilon}.\]
\end{proof}

\begin{lemma}\label{lem:divisor-sums}
For $Z\ge2$, we have
\[
 \sum_{d\le Z}2^{\omega(d)}\ll Z\log Z,
 \qquad
 \sum_{d>Z}\frac{2^{\omega(d)}}{d^2}
 \ll\frac{\log Z}{Z}.
\]
Consequently, for every fixed $j\ge0$,
\[
 \sum_{d\le Z}2^{\omega(d)}d^j
 \ll Z^{j+1}\log Z.
\]
\end{lemma}

\begin{proof}
Using the identity
$2^{\omega(d)}=\sum_{m\mid d}\mu^2(m)$,
we obtain
\[
\begin{aligned}
 \sum_{d\le Z}2^{\omega(d)}
 &=\sum_{d\le Z}\sum_{m\mid d}\mu^2(m)\\
 &=\sum_{m\le Z}\mu^2(m)
   \left\lfloor\frac Zm\right\rfloor\\
 &\ll Z\sum_{m\le Z}\frac1m
 \ll Z\log Z.
\end{aligned}
\]
Put $A(u)=\sum_{d\le u}2^{\omega(d)}$.
Then $A(u)\ll u\log u$.
For $T>Z$, partial summation gives
\[
\begin{aligned}
 \sum_{Z<d\le T}\frac{2^{\omega(d)}}{d^2}
 &=
 \frac{A(T)}{T^2}
 -\frac{A(Z)}{Z^2}+2\int_Z^T\frac{A(u)}{u^3}\,du.
\end{aligned}
\]
Since
$A(T)/T^2\ll(\log T)/T\to0$
as $T\to\infty$, we obtain
\[
\begin{aligned}
 \sum_{d>Z}\frac{2^{\omega(d)}}{d^2}
 &=
 -\frac{A(Z)}{Z^2}
 +2\int_Z^\infty\frac{A(u)}{u^3}\,du\\
 &\le
 2\int_Z^\infty\frac{A(u)}{u^3}\,du\\
 &\ll \frac{\log Z}{Z}.
\end{aligned}
\]
Finally, for $j\ge0$, we have
$d^j\le Z^j$ whenever $d\le Z$.
Therefore,
\[
\begin{aligned}
 \sum_{d\le Z}2^{\omega(d)}d^j
 &\le Z^j\sum_{d\le Z}2^{\omega(d)}\\
 &\ll Z^{j+1}\log Z.
\end{aligned}
\]
\end{proof}

\begin{theorem}\label{thm:average}
Let $B\in\{\Xred,\Xex\}$.  For every $\varepsilon>0$,
\[
\begin{aligned}
 &\sideset{}{^\flat}\sum_{\Delta\le X}
 \#\bigl\{I \text{ a reduced fractional ideal of } \Q(\sqrt{\Delta}) :
 \xi_I\in B\bigr\} \\
 &\qquad=
 \frac{\mathfrak S}{3}\,\nu(B)\,X^{3/2}
 +O_\varepsilon\bigl(X^{5/4+\varepsilon}\bigr),
\end{aligned}
\]
where
$\mathfrak S=\frac{11}{16}\prod_{p>2}
(1-p^{-2}-p^{-3}+p^{-4})$.
In particular, Theorem~\ref{thm:intro-average} holds.
\end{theorem}

\begin{proof}
Put $\mathcal K=\mathcal K_B$ and
$t=X^{1/2}$.
Let $\mu$ denote the M\"obius function.
By the characterization of fundamental
discriminants, the condition on $q(x)$
is equivalent to
$q(x)\bmod16\in\mathcal F_{16}$
and $p^2\nmid q(x)$ for every odd
prime $p$.
The latter condition has the indicator
\[
 \mathbf 1_{\{p^2\nmid q(x)
 \text{ for every odd prime }p\}}
 =
 \sum_{\substack{d\ge1,\ d\ {\rm odd}\\
 d^2\mid q(x)}}\mu(d).
\]
Hence, by \eqref{eq:sum-as-lattice-count},
the sum in the theorem equals
\[
 \sum_{x\in\Z^3\cap t\mathcal K}
 \mathbf 1_{\{q(x)\bmod16\in\mathcal F_{16}\}}
 \sum_{\substack{d\ge1,\ d\ {\rm odd}\\
 d^2\mid q(x)}}\mu(d).
\]
Interchanging the sums, we obtain
\[
 \sum_{\substack{d\ge1\\d\ {\rm odd}}}
 \mu(d)N(d),
\]
where
\[
 N(d)=
 \#\bigl\{
 x\in\Z^3\cap t\mathcal K:
 q(x)\bmod16\in\mathcal F_{16},\
 d^2\mid q(x)
 \bigr\}.
\]
Since $\mu(d)=0$ unless $d$ is
squarefree, we may restrict the sum
to odd squarefree integers $d$.

The two congruence conditions depend
only on $x$ modulo $16$ and modulo
$d^2$, respectively.
Put
\[
 \varrho(d^2)
 =
 \#\bigl\{
 x\in(\Z/d^2\Z)^3:
 q(x)\equiv0\pmod{d^2}
 \bigr\}.
\]
Since $d$ is odd and squarefree,
the Chinese remainder theorem gives
exactly $\varrho_{16}\varrho(d^2)$
admissible residue classes modulo
$16d^2$, and
\[
 \varrho(d^2)
 =\prod_{p\mid d}\varrho(p^2).
\]
By Lemma~\ref{lem:local-densities},
for every odd prime $p$, we have
\[
\begin{aligned}
 \varrho(p^2)
 &=p^4+p^3-p^2\le 2p^4.
\end{aligned}
\]
Consequently,
\[
\begin{aligned}
 \varrho(d^2)
 &=\prod_{p\mid d}\varrho(p^2)\le\prod_{p\mid d}2p^4=2^{\omega(d)}d^4,
\end{aligned}
\]
where $\omega(d)$ denotes the number
of distinct prime divisors of $d$.

For each admissible residue class
$x_0$ modulo $16d^2$, the points
$x\in t\mathcal K$ satisfying
$x\equiv x_0\pmod{16d^2}$
correspond to
\[
 \Z^3\cap
 \frac{t\mathcal K-x_0}{16d^2}.
\]
We apply the three-dimensional form
of Davenport's lemma
\cite{Davenport1951}.
For a bounded semialgebraic set
$S\subset\R^3$ of bounded complexity,
it gives
\[
 \#(S\cap\Z^3)
 =
 \vol(S)+O\bigl(V_2(S)+V_1(S)+1\bigr),
\]
where $V_j(S)$ denotes the sum of
the $j$-dimensional volumes of the
coordinate projections of $S$.
The implied constant depends only
on the complexity of $S$.
Take
\[
 S=\frac{t\mathcal K-x_0}{16d^2}
 =\frac{t}{16d^2}\mathcal K
 -\frac{x_0}{16d^2}.
\]
Since translation preserves volumes
and scaling by a factor $\lambda$
multiplies $j$-dimensional volumes
by $\lambda^j$, we have
\[
 \vol(S)
 =\frac{t^3}{(16d^2)^3}\vol(\mathcal K),
 \qquad
 V_j(S)
 =\left(\frac{t}{16d^2}\right)^j
 V_j(\mathcal K)
 \qquad (j=1,2).
\]
Since $V_j(\mathcal K)=O(1)$,
it follows that
\[
 V_2(S)\ll\frac{t^2}{d^4},
 \qquad
 V_1(S)\ll\frac{t}{d^2}.
\]
Applying Davenport's lemma, we obtain
\[
 \#\left(
 \Z^3\cap\frac{t\mathcal K-x_0}{16d^2}
 \right)
 =
 \frac{\vol(\mathcal K)t^3}{(16d^2)^3}
 +
 O\left(
 \frac{t^2}{d^4}
 +\frac{t}{d^2}+1
 \right).
\]
Summing over the
$\varrho_{16}\varrho(d^2)$
admissible residue classes and using
$\varrho(d^2)\le2^{\omega(d)}d^4$,
we obtain
\[
N(d)
 =
 \frac{\varrho_{16}\varrho(d^2)}
 {(16d^2)^3}
 \vol(\mathcal K)t^3+
 O\left(
 2^{\omega(d)}d^4
 \left(
 \frac{t^2}{d^4}
 +\frac{t}{d^2}+1
 \right)
 \right).
\]
Since $t=X^{1/2}$, this becomes
\[
 N(d)
 =
 \frac{\varrho_{16}\varrho(d^2)}
 {(16d^2)^3}
 \vol(\mathcal K)X^{3/2}+
 O\left(
 2^{\omega(d)}
 (X+X^{1/2}d^2+d^4)
 \right).
\]

We first consider $d\le X^{1/4}$.
By Lemma~\ref{lem:divisor-sums},
the total error for $d\le X^{1/4}$
is
\[
\begin{aligned}
 O&\Biggl(
 X\sum_{d\le X^{1/4}}2^{\omega(d)}
 +X^{1/2}\sum_{d\le X^{1/4}}
 2^{\omega(d)}d^2+\sum_{d\le X^{1/4}}
 2^{\omega(d)}d^4
 \Biggr)\\
 &=O\left(
 \left(
 X^{5/4}
 +X^{1/2+3/4}
 +X^{5/4}
 \right)\log X
 \right)\\
 &=O(X^{5/4}\log X).
\end{aligned}
\]
We next extend the main term to all
odd squarefree integers $d$.
Since
$\frac{\varrho(d^2)}{d^6}
 \le\frac{2^{\omega(d)}}{d^2}$
for such $d$,
Lemma~\ref{lem:divisor-sums} gives
\[
\begin{aligned}
 X^{3/2}
 \sum_{\substack{d>X^{1/4}d\ {\rm odd,\ squarefree}}}
 \frac{\varrho(d^2)}{d^6}&\ll
 X^{3/2}
 \sum_{d>X^{1/4}}
 \frac{2^{\omega(d)}}{d^2}\\
 &\ll X^{5/4}\log X.
\end{aligned}
\]
Thus the truncated main term
may be replaced by
\[
 \frac{\varrho_{16}}{16^3}
 \sum_{\substack{d\ge1\\
 d\ {\rm odd,\ squarefree}}}
 \mu(d)\frac{\varrho(d^2)}{d^6}
 \vol(\mathcal K)X^{3/2}
\]
with an error $O(X^{5/4}\log X)$.
The series converges absolutely.
Moreover, the Chinese remainder
theorem gives
$\varrho(d^2)
 =\prod_{p\mid d}\varrho(p^2)$
for every odd squarefree $d$.
Hence
\[
 \sum_{\substack{d\ge1\\
 d\ {\rm odd,\ squarefree}}}
 \mu(d)\frac{\varrho(d^2)}{d^6}
 =
 \prod_{p>2}
 \left(1-\frac{\varrho(p^2)}{p^6}\right).
\]
By Lemmas~\ref{lem:volume}
and~\ref{lem:local-densities},
the completed main term is
\[
 \frac{\mathfrak S}{3}
 \nu(B)X^{3/2}.
\]

It remains to estimate the contribution
of $d>X^{1/4}$ in the original
M\"obius sum.
Since $|\mu(d)|\le1$, it suffices
to bound
$\sum_{d>X^{1/4}}N(d)$.
If $x$ is counted by $N(d)$,
then $q(x)=nd^2$ for some
positive integer $n\le X/d^2$.
Hence Lemma~\ref{lem:crude} gives
\[
\begin{aligned}
 N(d)
 &\le
 \sum_{n\le X/d^2}\kappa^*(nd^2)\\
 &\ll_\varepsilon
 \sum_{n\le X/d^2}
 (nd^2)^{1/2+\varepsilon}\\
 &\ll_\varepsilon
 d^{1+2\varepsilon}
 \left(\frac{X}{d^2}\right)^{3/2+\varepsilon}\\
 &=\frac{X^{3/2+\varepsilon}}{d^2}.
\end{aligned}
\]
Consequently,
\[
\begin{aligned}
 \sum_{d>X^{1/4}}N(d)
 &\ll_\varepsilon
 X^{3/2+\varepsilon}
 \sum_{d>X^{1/4}}\frac1{d^2}\\
 &\ll_\varepsilon
 X^{5/4+\varepsilon}.
\end{aligned}
\]
Combining these estimates proves
the asserted asymptotic formula.
Theorem~\ref{thm:intro-average}
then follows from
Lemma~\ref{lem:measures}.
\end{proof}

\section{Equidistribution and individual fields}\label{sec:equidistribution}

In this section we prove Theorems~\ref{thm:intro-individual} and
\ref{thm:intro-defect}.  The deep input is Duke's theorem.  The
passage from closed geodesics to continued fractions through a
cross-section of the geodesic flow is classical
\cite{Arnoux1994,Mangual2008,Pollicott1986,Series1985}.  We include
the details because we need all ideal classes of a fixed discriminant
simultaneously, sets with curved boundary such as $\Xex$, and control
of the noncompact part of $\Xred$.

For orientation, the geometry below may be summarized as follows.  A
reduced pair $\xi=(w_+,w_-)\in\Xred$ determines an oriented geodesic
in the upper half-plane with endpoints $w_-$ and $w_+$.  We record
this geodesic when it crosses the imaginary axis.  Away from the null set of integer forward endpoints, the continued-
fraction reduction map $\Phi$ is realized as the return map to
this cross-section, with return time $\Theta$.  For a fixed
fundamental discriminant, the $\Phi$-orbits are the reduction cycles,
one for each ideal class, and each cycle traces one closed geodesic.
Duke's theorem equidistributes the union of these closed geodesics;
the cross-section then transfers this equidistribution to the reduced
pairs $\xi_I$.

\subsection{The geodesic flow and a cross-section}
Let $\HH=\{x+iy:y>0\}$ be the upper half-plane with hyperbolic metric
$ds=|dz|/\operatorname{Im}z$.
Its geodesics are the vertical lines and the Euclidean semicircles
orthogonal to the real axis.  
Let $T^1\HH$ be the unit tangent bundle
of $\HH$ and $(a_t)_{t\in\R}$ the geodesic flow. 
More explicitly, if
$v\in T^1\HH$ and $\gamma_v$ is the unit-speed geodesic satisfying
$\dot\gamma_v(0)=v$, then
$a_t(v)=\dot\gamma_v(t)$.
Thus $a_t$ moves the unit tangent vector $v$ a hyperbolic distance
$t$ along the geodesic determined by $v$.

The group $\Gamma=\PGL_2(\Z)$ acts on $\HH$ by isometries.  A matrix
$\left(\begin{smallmatrix}\alpha&\beta\\ \gamma&\delta\end{smallmatrix}\right)$
acts by $z\mapsto(\alpha z+\beta)/(\gamma z+\delta)$ when its
determinant is $1$, and by
$z\mapsto(\alpha\bar z+\beta)/(\gamma\bar z+\delta)$ when its
determinant is $-1$.  The induced action on $T^1\HH$ commutes with
the geodesic flow and is properly discontinuous.  The vectors with
nontrivial stabilizer form a null set: such a vector is tangent to a
reflection geodesic for $\Gamma$, and these geodesics have endpoints
in $\Q\cup\{\infty\}$.

Put $Y=\Gamma\backslash T^1\HH$ and let $\pi:T^1\HH\to Y$ be the
projection.  We use the Liouville measure on $T^1\HH$ normalized by
\[
 d\Lv=\frac{dx\,dy\,d\theta}{y^2},
 \qquad 0\le\theta<2\pi,
\]
where $\theta$ denotes the direction angle of a unit tangent vector.
We use the same symbol $\Lv$ for the induced finite quotient measure
on $Y$.
Every $v\in T^1\HH$ is tangent to an oriented geodesic with backward
and forward endpoints $w_-(v)\ne w_+(v)$ in $\R\cup\{\infty\}$.
For geodesics with finite endpoints, choose a reference point depending
smoothly on $(w_+,w_-)$.  Let $t$ be the signed hyperbolic arclength
from this point, increasing toward $w_+$.  The resulting Hopf
coordinates $(w_+,w_-,t)$ identify a unit tangent vector by its
geodesic and its position along it.  In these coordinates, the flow
$a_s$ fixes both endpoints and replaces $t$ by $t+s$.
On the region where $(w_+,w_-)\in\Xred$, the change to Hopf
coordinates gives
\begin{equation}\label{eq:hopf}
 d\Lv=2\,\frac{dw_+\,dw_-\,dt}{(w_+-w_-)^2}
 =2\,d\nu\,dt.
\end{equation}
See, for example, \cite[Ch.~9]{EinsiedlerWard2011}.

We now make the cross-section concrete.  Let
$\xi=(w_+,w_-)\in\Xred$.  Since $w_-<0<w_+$, the oriented geodesic
$\mathcal C_\xi$ from $w_-$ to $w_+$ is the upper semicircle with
centre $c=(w_++w_-)/2$ and radius $R=(w_+-w_-)/2$; equivalently,
\[
 (x-c)^2+y^2=R^2.
\]
It crosses the imaginary axis at
$z_0=i\sqrt{-w_+w_-}$. 
Define
$\sigma:\Xred\longrightarrow T^1\HH$
by letting $\sigma(\xi)$ be the unit tangent vector to
$\mathcal C_\xi$ at $z_0$, oriented from $w_-$ toward $w_+$.
For these geodesics, we choose $z_0$ as the reference point $t=0$ in
the Hopf coordinates.  Thus $\Sigma=\sigma(\Xred)$ is the
cross-section that we use below.  The map
\begin{equation}\label{eq:flow-box-map}
 \mathcal F:\Xred\times\R\longrightarrow T^1\HH,
 \qquad (\xi,t)\longmapsto a_t(\sigma(\xi)),
\end{equation}
is a diffeomorphism onto an open set, and \eqref{eq:hopf} gives
$\mathcal F^*\Lv=2\,\nu\otimes dt$.

Before stating the return-map lemma, we describe one reduction step.
Assume $w_+\notin\Z$, as is always the case for the quadratic
pairs $\xi_I$.  Put $n=\lfloor w_+\rfloor$.  Since $n<w_+<n+1$,
the same semicircle
meets the vertical line $\operatorname{Re}z=n$ at
$z_n=n+i\sqrt{(n-w_-)(w_+-n)}$.

Let $g_n\in\Gamma$ be the class of
$\left(\begin{smallmatrix}0&1\\1&-n\end{smallmatrix}\right)$, so that
$g_n(z)=1/(\bar z-n)$.  On the real boundary,
$g_n(w_+)=1/(w_+-n)$ and $g_n(w_-)=1/(w_--n)$, while
$g_n(n+iy)=i/y$.  Thus $g_n$ sends the line
$\operatorname{Re}z=n$ to the imaginary axis and sends the two
endpoints to the pair
\[\Phi(\xi)=(1/(w_+-n),1/(w_--n)).\]

Geometrically, one reduction step is as follows: starting from the
crossing point $z_0$, follow $\mathcal C_\xi$ until the crossing point
$z_n$, and then apply $g_n$ to send $z_n$ back to the imaginary axis.
Since $g_n$ is an isometry of $\HH$, its differential
$dg_n:T^1\HH\to T^1\HH$ sends the corresponding unit tangent vector
to the unit tangent vector at the new crossing point.

\begin{lemma}\label{lem:return}
Let $\xi=(w_+,w_-)\in\Xred$ with
$w_+\notin\Z$, and put
$n=\lfloor w_+\rfloor$.
Let $\Theta(\xi)$ denote the arclength
along $\mathcal C_\xi$ from $z_0$ to $z_n$.
Then $\Phi(\xi)\in\Xred$, and
\[
 dg_n\bigl(
 a_{\Theta(\xi)}(\sigma(\xi))
 \bigr)
 =\sigma(\Phi(\xi)).
\]
The return time is given by
\begin{equation}\label{eq:return-time}
 \Theta(\xi)
 =\frac12\log\frac{w_+}{w_+-n}
 +\frac12\log\frac{n-w_-}{-w_-}.
\end{equation}
In particular, we have
\begin{align*}
 \Theta(\xi)
 &>\frac12\log2,
 &
 \Theta(\xi)
 &>\log w_+,\\
 \Theta(\xi)
 &>\frac12\log\frac1{|w_-|},
 &
 \Theta(\xi)
 &>\frac12\log\frac1{w_+-n}.
\end{align*}
\end{lemma}

\begin{proof}
Since $0<w_+-n<1$ and $w_--n<-1$,
we have
\[
 \frac1{w_+-n}>1,
 \qquad
 -1<\frac1{w_--n}<0.
\]
Thus $\Phi(\xi)\in\Xred$.
By the definition of $g_n$, we have
\[
 g_n(w_\pm)=\frac1{w_\pm-n},
 \qquad
 g_n(n+iy)=\frac{i}{y}.
\]
Since $g_n$ is an isometry of $\HH$,
it maps $\mathcal C_\xi$ to the
oriented geodesic associated with
$\Phi(\xi)$.
Moreover, $g_n(z_n)$ lies on the
imaginary axis.
As this geodesic crosses the
imaginary axis at a unique point,
$g_n(z_n)$ is the base point of
$\sigma(\Phi(\xi))$.
The differential $dg_n$ preserves
the hyperbolic norm and maps the
oriented tangent direction of
$\mathcal C_\xi$ to that of
$g_n(\mathcal C_\xi)$.
Since $a_{\Theta(\xi)}(\sigma(\xi))$
is the unit tangent vector at $z_n$,
we obtain
\[
 dg_n\bigl(
 a_{\Theta(\xi)}(\sigma(\xi))
 \bigr)
 =\sigma(\Phi(\xi)).
\]

We next compute the arclength.
The M\"obius transformation
\[
 \psi(z)=-\frac1{z-w_+}
\]
maps $\mathcal C_\xi$ to a vertical
geodesic oriented upwards.
Along this geodesic, hyperbolic
arclength is measured by the change
in $\log\operatorname{Im}z$.
Since
\[
 \operatorname{Im}\psi(z)
 =\frac{\operatorname{Im}z}{|z-w_+|^2},
\]
and
\begin{align*}
 |z_0-w_+|^2
 &=w_+(w_+-w_-),\\
 |z_n-w_+|^2
 &=(w_+-n)(w_+-w_-),
\end{align*}
we obtain
\[
\begin{aligned}
 \Theta(\xi)
 &=\log
 \frac{\operatorname{Im}\psi(z_n)}
 {\operatorname{Im}\psi(z_0)}\\
 &=\log
 \frac{
 w_+\sqrt{(n-w_-)(w_+-n)}
 }{
 (w_+-n)\sqrt{-w_+w_-}
 }\\
 &=\frac12\log\frac{w_+}{w_+-n}
 +\frac12\log\frac{n-w_-}{-w_-}.
\end{aligned}
\]
This proves \eqref{eq:return-time}.

Finally, put
\[
 \alpha=\frac{w_+}{w_+-n},
 \qquad
 \beta=1+\frac{n}{|w_-|}.
\]
Since $0<w_+-n<1$, we have
\[
 \alpha>
 \max\left\{
 1,w_+,\frac1{w_+-n}
 \right\}.
\]
Moreover, $n\ge1$ and $|w_-|<1$
imply
\[
 \beta>n+1>w_+,
\]
and hence
\[
 \beta>
 \max\left\{
 2,w_+,\frac1{|w_-|}
 \right\}.
\]
Consequently,
\[
 \alpha\beta>
 \max\left\{
 2,w_+^2,\frac1{|w_-|},
 \frac1{w_+-n}
 \right\}.
\]
The four lower bounds follow from
\[
 \Theta(\xi)=\frac12\log(\alpha\beta).
\]
\end{proof}

On $\Xred_\Delta$, the map $\Phi$ is the usual continued-fraction reduction map. Its periodic orbits are the classical reduction cycles associated with ordinary ideal classes \cite{HalterKoch2013,KaplanWilliams1990}.  Write one such cycle as
\begin{equation}\label{eq:cycle-data}
 I_i=I(A_i,b_i)=\Z+\Z f_i,\qquad
 a_i=\lfloor f_i\rfloor,\qquad
 f_{i+1}=(f_i-a_i)^{-1},
\end{equation}
where $0\le i<\ell(C)$, indices are taken modulo $\ell(C)$, and $C$ is the corresponding ideal class.  The classical regulator identity is
\begin{equation}\label{eq:cycle-regulator}
 \sum_{i=0}^{\ell(C)-1}\log f_i=R_K.
\end{equation}
Lemma~\ref{lem:return} identifies successive reduction steps with returns of the same geodesic to the cross-section.  We next identify the resulting closed orbits and their lengths.

\begin{lemma}\label{lem:cycles}
Let $\Delta$ be a positive fundamental
discriminant and put $K=\Q(\sqrt\Delta)$.
Let $h=h(\Delta)$ be the class number
of $K$, and let $\varepsilon_K>1$ be
a fundamental unit.
Write $R_K=\log\varepsilon_K$ for
the regulator of $K$.

\begin{enumerate}[label=\textup{(\roman*)}]
\item
For every ideal class $C$ of $K$,
we have
\[
 \sum_{\substack{I\in C\\
 I\ \mathrm{reduced}}}
 \Theta(\xi_I)=2R_K.
\]
Consequently,
$\sum_{\xi\in\Xred_\Delta}
 \Theta(\xi)=2hR_K$.

\item
For any reduced ideal $I_0\in C$,
the curve
\[
 t\longmapsto
 \pi(a_t\sigma(\xi_{I_0}))
\]
is periodic with primitive period
$2R_K$.
Its image $O_C$ is the closed orbit
corresponding to the $\Gamma$-orbit
of the oriented geodesic from
$\bar f_0$ to $f_0$.
Distinct ideal classes determine
distinct closed orbits.

\item
Choose a reduced ideal $I_0^C$
in each ideal class $C$.
Define a measure $\Lambda_\Delta$
on $Y$ by
\[
 \Lambda_\Delta(\Psi)
 =
 \sum_C\int_0^{2R_K}
 \Psi\bigl(
 \pi(a_t\sigma(\xi_{I_0^C}))
 \bigr)\,dt.
\]
Then
\[
 \Lambda_\Delta(\Psi)
 =
 \sum_{\xi\in\Xred_\Delta}
 \int_0^{\Theta(\xi)}
 \Psi\bigl(\pi(a_t\sigma(\xi))\bigr)\,dt.
\]
In particular,
$\Lambda_\Delta(Y)=2hR_K$.
\end{enumerate}
\end{lemma}

\begin{proof}
\textup{(i)}
Fix an ideal class $C$, and let
$I_0,\ldots,I_{\ell-1}$ be its complete reduction cycle.
We use the notation of \eqref{eq:cycle-data}, with $I_\ell=I_0$.
By \eqref{eq:return-time}, we have
\[
 \Theta(\xi_{I_i})
 =
 \frac12\log\frac{f_i}{f_i-a_i}
 +
 \frac12\log\frac{a_i-\bar f_i}{-\bar f_i}.
\]
The continued-fraction relations give
$f_i-a_i=\frac{1}{f_{i+1}}$ and $a_i-\bar f_i=\frac{1}{|\bar f_{i+1}|}$.
Since $\bar f_i<0$, substituting these identities yields
\[
\begin{aligned}
 \Theta(\xi_{I_i})
 &=
 \frac12\log f_i
 +\frac12\log f_{i+1}
 -\frac12\log|\bar f_i|
 -\frac12\log|\bar f_{i+1}|.
\end{aligned}
\]
Summing cyclically and using $f_\ell=f_0$, we obtain
\[\sum_{i=0}^{\ell-1}\Theta(\xi_{I_i})
 =
 \sum_{i=0}^{\ell-1}\log f_i
 -
 \sum_{i=0}^{\ell-1}\log|\bar f_i|.\]
We consider the matrices
$A_i=
 \begin{pmatrix}
  a_i&1\\
  1&0
 \end{pmatrix}$ and $M=A_0A_1\cdots A_{\ell-1}$.
The continued-fraction relation
$f_i=a_i+1/f_{i+1}$ gives
$\begin{aligned}
 A_i
 \begin{pmatrix}f_{i+1}\\1\end{pmatrix}
 &=
 \begin{pmatrix}
  a_i f_{i+1}+1\\
  f_{i+1}
 \end{pmatrix}
 =
 f_{i+1}
 \begin{pmatrix}f_i\\1\end{pmatrix}.
\end{aligned}$
Applying Galois conjugation, we similarly obtain
$A_i
 \begin{pmatrix}\bar f_{i+1}\\1\end{pmatrix}
 =
 \bar f_{i+1}
 \begin{pmatrix}\bar f_i\\1\end{pmatrix}$.
Iterating these identities and using the periodicity
$f_\ell=f_0$, we obtain
\[M\begin{pmatrix}f_0\\1\end{pmatrix}
=
 \left(\prod_{i=0}^{\ell-1}f_i\right)
 \begin{pmatrix}f_0\\1\end{pmatrix},\qquad M\begin{pmatrix}\bar f_0\\1\end{pmatrix}
 =
 \left(\prod_{i=0}^{\ell-1}\bar f_i\right)
 \begin{pmatrix}\bar f_0\\1\end{pmatrix}.\]
Since $f_0\ne\bar f_0$, these are two linearly
independent eigenvectors of $M$. Thus the corresponding
eigenvalues are the products of the $f_i$ and of their
conjugates.
Since $\det A_i=-1$, we have $|\det M|=1$.
Taking the product of the two eigenvalues gives
$\left|
 \prod_{i=0}^{\ell-1}f_i
 \prod_{i=0}^{\ell-1}\bar f_i
 \right|=1$.
Consequently,
$\sum_{i=0}^{\ell-1}\log|\bar f_i|
 =
 -\sum_{i=0}^{\ell-1}\log f_i$.
It follows that
\[
 \sum_{i=0}^{\ell-1}\Theta(\xi_{I_i})
=
 \sum_{i=0}^{\ell-1}\log f_i
 -
 \sum_{i=0}^{\ell-1}\log|\bar f_i|=
 2R_K,
\]
where the last equality follows from
\eqref{eq:cycle-regulator}.
Summing this identity over all $h$ ideal classes proves the
second assertion.

\medskip
\noindent
\textup{(ii)}
We first show that the reduction cycle gives a periodic orbit
on $Y$.
For the reduction cycle, write
$v_i=\sigma(\xi_{I_i})$, $g_i=g_{a_i}$, and
$\Theta_i=\Theta(\xi_{I_i})$.
By Lemma~\ref{lem:return}, each reduction step satisfies
$v_{i+1}=dg_i(a_{\Theta_i}v_i)$.
Set $G_0=1$ and $S_0=0$, and for $1\leq i\leq\ell$ let
$G_i=g_{i-1}\cdots g_0$ and $S_i=\sum_{j=0}^{i-1}\Theta_j$.
Since each $g_i$ is an isometry, its differential commutes
with the geodesic flow. Using this commutation relation
and the flow identity $a_s\circ a_t=a_{s+t}$,
we obtain inductively
\[
\begin{aligned}
 v_{i+1}
 &=dg_i(a_{\Theta_i}v_i)\\
 &=dg_i\,a_{\Theta_i}\,dG_i\,a_{S_i}v_0\\
 &=dg_i\,dG_i\,a_{\Theta_i}\,a_{S_i}v_0\\
 &=dG_{i+1}(a_{S_{i+1}}v_0).
\end{aligned}
\]
Consequently,
$v_\ell=dG_\ell(a_{S_\ell}v_0)$.
By part~\textup{(i)}, we have $S_\ell=2R_K$.
Moreover, the reduction cycle returns to its initial ideal,
so $v_\ell=v_0$. Hence
\[dG_\ell(a_{2R_K}v_0)=v_0.\]
Passing to the quotient by $\Gamma$, we obtain
$\pi(a_{2R_K}v_0)=\pi(v_0)$.
Thus the projected geodesic-flow orbit is periodic,
with $2R_K$ as one of its periods.
We next prove that this period is primitive.
Suppose that $T>0$ is any period of the projected orbit.
Then there exists $\gamma\in\Gamma$ such that
$d\gamma(v_0)=a_Tv_0$.
Let $c$ be the unit-speed geodesic with initial
tangent vector $v_0$.
By uniqueness of geodesics, we have
$\gamma(c(t))=c(t+T)$, where $t\in\R$.
Thus $\gamma$ preserves the oriented geodesic $c$
and fixes its two endpoints $\bar f_0$ and $f_0$
individually.
To describe this stabilizer, let
$L=\Z+\Z f_0$.
Suppose that
$A=\begin{pmatrix}a&b\\c&d\end{pmatrix}
\in\GL_2(\Z)$
fixes $f_0$. Then
\[
af_0+b=f_0(cf_0+d).
\]
Putting $u=cf_0+d$, we have
$uf_0=af_0+b$. Since $\det A=\pm1$,
\[
\begin{aligned}
uL
&=\Z u+\Z uf_0\\
&=\Z(cf_0+d)+\Z(af_0+b)=L.
\end{aligned}
\]
Thus $u$ is a unit of the multiplier ring
$(L:L)$.
Conversely, let $u$ be a unit of $(L:L)$.
Then $uL=L$, so there exist integers
$a,b,c,d$ such that
\[
uf_0=af_0+b,\qquad u=cf_0+d.
\]
Since $uL=L$, the corresponding matrix
$A$ belongs to $\GL_2(\Z)$, and the above
identities show that $A$ fixes $f_0$.
Since $L$ is a fractional $\OO_K$-ideal, $(L:L)=\OO_K$.
Consequently, the stabilizer is isomorphic to
\[(L:L)^\times/\{\pm1\}
=\OO_K^\times/\{\pm1\}.\]
Let $u=\varepsilon_K^m$.
The matrix corresponding to multiplication by $u$
has eigenvalues $u$ and $\bar u$.
After sending the endpoints of $c$ to $0$ and
$\infty$, the corresponding isometry acts on
the imaginary axis by
$iy\longmapsto i\left|\frac{u}{\bar u}\right|y
=iu^2y$,
where the last equality follows from
$|u\bar u|=1$.
Since the hyperbolic line element on the
imaginary axis is $dy/y$, the signed
translation distance is
\[
2\log u=2mR_K.
\]
Hence every positive period is a positive
integral multiple of $2R_K$.
Since $2R_K$ is already a period,
it is primitive.

It remains to show that distinct ideal
classes give distinct closed orbits.
Suppose that the oriented geodesics
associated with
$I=\Z+\Z f$ and $J=\Z+\Z f'$
are $\Gamma$-equivalent.
Then
\[
 f'=\frac{af+b}{cf+d}
\]
for some
$\begin{pmatrix}a&b\\c&d\end{pmatrix}
\in\GL_2(\Z)$.
Consequently,
\[
\begin{aligned}
 J
 &=\Z+\Z f'\\
 &=(cf+d)^{-1}
   \bigl(\Z(cf+d)+\Z(af+b)\bigr)\\
 &=(cf+d)^{-1}I.
\end{aligned}
\]
Thus $I$ and $J$ belong to the
same ordinary ideal class.

\medskip
\noindent
\textup{(iii)}
Fix an ideal class $C$ and use the notation
of part~\textup{(ii)}.
Since $v_i=dG_i(a_{S_i}v_0)$ with $G_i\in\Gamma$,
we have
$\pi(a_tv_i)=\pi(a_{S_i+t}v_0)$.
By part~\textup{(i)}, $S_\ell=2R_K$.
Thus
\[
\begin{aligned}
\int_0^{2R_K}\Psi(\pi(a_tv_0))\,dt
&=
\sum_{i=0}^{\ell-1}
\int_{S_i}^{S_{i+1}}
\Psi(\pi(a_tv_0))\,dt\\
&=
\sum_{i=0}^{\ell-1}
\int_0^{\Theta_i}
\Psi(\pi(a_tv_i))\,dt.
\end{aligned}
\]
Summing over all ideal classes gives the
claimed identity for $\Lambda_\Delta$.
Taking $\Psi=1$, we obtain
$\Lambda_\Delta(Y)=2hR_K$.
\end{proof}

\subsection{Duke's theorem}

By Lemma~\ref{lem:cycles}, each positive
fundamental discriminant determines
a finite collection of closed geodesics,
one for each ideal class.
We use Duke's equidistribution theorem
to study their distribution as
$\Delta\to\infty$.

Let
\[
 \Gamma_1=\PSL_2(\Z),
 \qquad
 Y_1=\Gamma_1\backslash T^1\HH,
\]
and let $\Lv_{Y_1}$ be the Liouville
measure on $Y_1$.

For a positive nonsquare discriminant
$\Delta$, the $\mathrm{SL}_2(\Z)$-equivalence
classes of primitive indefinite
binary quadratic forms of
discriminant $\Delta$ correspond
to oriented closed geodesics in $Y_1$.
Let $\Lambda^{(1)}_\Delta$ denote
the sum of the length measures
on these orbits, each counted once.

\begin{theorem}[Duke]\label{thm:duke}
As $\Delta\to\infty$ through positive
fundamental discriminants, we have
\[
 \frac{\Lambda^{(1)}_\Delta}
 {\Lambda^{(1)}_\Delta(Y_1)}
 \xrightarrow{\mathrm{weak}-*}
 \frac{\Lv_{Y_1}}{\Lv_{Y_1}(Y_1)}.
\]
\end{theorem}

Duke \cite{Duke1988} established
the equidistribution of closed geodesics
on the modular surface.
The formulation on the unit tangent
bundle used here is given in
\cite[Theorem~1.3]
{EinsiedlerLindenstraussMichelVenkatesh2012}.

\begin{corollary}\label{cor:duke-Y}
For positive fundamental discriminants $\Delta\to\infty$,
the measures $\Lambda_\Delta/(2hR_K)$ converge weak-$*$ to
$\Lv/\Lv(Y)$ on $Y$.
Consequently,
\[
 \frac{\Lambda_\Delta(\Psi)}{2hR_K}
 \longrightarrow
 \frac{1}{\Lv(Y)}\int_Y\Psi\,d\Lv
\]
for every bounded Borel function $\Psi$ on $Y$ that is continuous
outside an $\Lv$-null set.
\end{corollary}

\begin{proof}
Let $p:Y_1\to Y$ be the projection.  Write $\Lv_{Y_1}$ and
$\Lv_Y$ for the respective Liouville measures induced by the
hyperbolic metric.  Since $\Gamma_1$ has index two in $\Gamma$,
these quotient measures satisfy $p_*\Lv_{Y_1}=2\Lv_Y$ and
$\Lv_{Y_1}(Y_1)=2\Lv_Y(Y)$.  In particular, $p$ sends the normalized
Liouville measure on $Y_1$ to the normalized Liouville measure on $Y$.
Let $\iota$ be the involution of $Y_1$ induced by an element of
$\Gamma\setminus\Gamma_1$.
A fixed point of $\iota$ is tangent to a reflection geodesic.
Since every reflection geodesic in $\Gamma$ has endpoints in
$\mathbb Q\cup\{\infty\}$, whereas geodesics of nonsquare
discriminant have irrational endpoints, no closed orbit of
nonsquare discriminant contains such a fixed point. Hence $p$ is exactly
two-to-one on the union of the closed orbits of discriminant
$\Delta$, and by Lemma~\ref{lem:cycles}(ii) this union is the inverse
image of $\bigcup_C O_C$.  Therefore
$p_*\Lambda^{(1)}_\Delta=2\Lambda_\Delta$, and after normalization
Duke's theorem gives the asserted weak-$*$ convergence.  The final
statement follows first for bounded continuous functions and then,
by the portmanteau theorem, for bounded Borel functions that are
continuous almost everywhere with respect to the limit measure.
\end{proof}

\subsection{Equidistribution of reduced pairs}

The remaining step is to pass from length measure on closed geodesics
to counting measure on their intersections with the cross-section.
For a relatively compact set $B\subset\Xred$, we thicken
$\sigma(B)$ for a short time in the flow direction and apply
Corollary~\ref{cor:duke-Y}.  We then control the noncompact ends of
$\Xred$ by the lower bounds for the return time in
Lemma~\ref{lem:return}.  This gives the following equidistribution of
reduced pairs.

\begin{proposition}\label{prop:equidistribution}
Let $B\subset\Xred$ be a Borel set whose boundary in $\Xred$ is
$\nu$-null.  As $\Delta\to\infty$ through positive fundamental
discriminants,
\[
 \frac{\#(\Xred_\Delta\cap B)}{\kappa(\Delta)}
 \longrightarrow\frac{\nu(B)}{\log2}.
\]
Moreover,
\[
 \frac{\kappa(\Delta)}{h(\Delta)R_K}
 \longrightarrow\frac{12\log2}{\pi^2}.
\]
\end{proposition}

\begin{proof}
\emph{Step 1. Relatively compact sets.}
Suppose first that the closure
$\overline B$ of $B$ in $\Xred$
is compact.
Define
\[
 \Psi_B(y)
 =4\,\#\left\{
 (\xi,s)\in B\times[0,\tfrac14)
 \mid \pi(\mathcal F(\xi,s))=y
 \right\}.
\]
Since $\mathcal F$ is injective and
$\mathcal F(\overline B\times[0,\frac14])$
is compact, proper discontinuity
of the $\Gamma$-action implies
that $\Psi_B$ is bounded and
compactly supported.
Away from the image of the flow-box
boundary and the locus of nontrivial
$\Gamma$-stabilizers, the map
$\pi\circ\mathcal F$ is a local
diffeomorphism.
Hence $\Psi_B$ is locally constant
outside these sets.
They are $\Lv$-null by
$\mathcal F^*\Lv=2\,\nu\otimes dt$,
the assumption $\nu(\partial B)=0$,
and the nullity of the stabilizer locus.
By the definition of $\Psi_B$ and
the flow-box measure formula, we obtain
\[
\begin{aligned}
 \int_Y\Psi_B\,d\Lv
 &=8\int_B\int_0^{1/4}
   ds\,d\nu(\xi)\\
 &=2\nu(B).
\end{aligned}
\]

We next evaluate $\Psi_B$ on
the closed geodesics of discriminant
$\Delta$.
Fix an ideal class $C$, and let
$I_0,\ldots,I_{\ell-1}$ be its
complete reduction cycle.
Put
\[
 S_0=0,\qquad
 S_j=\sum_{i=0}^{j-1}
 \Theta(\xi_{I_i})
 \quad (1\le j\le\ell),
\]
and write
\[
 y_t=\pi(a_t\sigma(\xi_{I_0}))
 \qquad (0\le t<2R_K).
\]
By Lemma~\ref{lem:return}, we have
$y_{S_j}=\pi(\sigma(\xi_{I_j}))$, where $0\le j<\ell$.
Moreover,
$S_\ell=2R_K$
by Lemma~\ref{lem:cycles}(i).
We show that these are precisely
the intersections of the closed
orbit with $\pi(\Sigma)$.
Suppose that
\[
 y_t=\pi(\sigma(\xi))
\]
for some $\xi=(w_+,w_-)\in\Xred$.
Then the oriented geodesic
associated with $\xi$ is
$\Gamma$-equivalent to that
associated with $I_0$.
Writing $I_0=\Z+\Z f_0$, we obtain
\[
 \Z+\Z w_+
 =(cf_0+d)^{-1}I_0
\]
for suitable integers $c,d$.
Since $\xi\in\Xred$, this
determines a reduced ideal
in the class $C$.
Hence $\xi=\xi_{I_j}$ for
some $0\le j<\ell$.
It follows that
\[
 y_t=y_{S_j}.
\]
Since $2R_K$ is the primitive
period and both $t$ and $S_j$
belong to $[0,2R_K)$,
we conclude that $t=S_j$.
Thus the intersections with
$\pi(\Sigma)$ in one primitive
period correspond exactly to
the reduced ideals in $C$.
By Lemma~\ref{lem:return},
$\Theta(\xi_{I_j})
 >\frac12\log2>\frac14$.
Therefore the intervals
\[
 [S_j,S_j+\tfrac14)
 \qquad (0\le j<\ell)
\]
are pairwise disjoint and
contained in $[0,2R_K)$.
Each reduced ideal $I_j$ with
$\xi_{I_j}\in B$ contributes
exactly one such interval.
Consequently,
\[
\begin{aligned}
 \int_0^{2R_K}\Psi_B(y_t)\,dt
 &=4\sum_{\substack{0\le j<\ell\\
                 \xi_{I_j}\in B}}
   \int_{S_j}^{S_j+1/4}dt\\
 &=\#\{I\in C:
 I\text{ is reduced},\
 \xi_I\in B\}.
\end{aligned}
\]
Summing over all ideal classes,
we obtain
\[
 \Lambda_\Delta(\Psi_B)
 =\#(\Xred_\Delta\cap B).
\]
Corollary~\ref{cor:duke-Y}
therefore yields
\begin{equation}\label{eq:step1}
 \frac{\#(\Xred_\Delta\cap B)}
 {2hR_K}
 \longrightarrow
 \frac{2\nu(B)}{\Lv(Y)}.
\end{equation}

\medskip
\noindent
\emph{Step 2. Tightness.}
We now control the reduced pairs
near the noncompact ends of $\Xred$.
Let $B'\subset\Xred$, and suppose
that $\Theta(\xi)\ge Z>0$
for every $\xi\in
B'\cap\Xred_\Delta$.
By Lemma~\ref{lem:cycles}(i),
\[
\begin{aligned}
 Z\,\#(\Xred_\Delta\cap B')
 &\le
 \sum_{\xi\in\Xred_\Delta}
 \Theta(\xi)\\
 &=2hR_K.
\end{aligned}
\]
Hence
$\#(\Xred_\Delta\cap B')
 \le\frac{2hR_K}{Z}$.
For $0<\eta\le1/16$, put
\[
 K_\eta=
 [1+\eta,1/\eta]\times
 [-1+\eta,-\eta].
\]
The complement
$\Xred\setminus K_\eta$
is covered by the four regions
\begin{align*}
 E_1&=\{\xi\in\Xred:
             w_+>1/\eta\},\\
 E_2&=\{\xi\in\Xred:
             -\eta<w_-<0\},\\
 E_3&=\{\xi\in\Xred:
             1<w_+<1+\eta\},\\
 E_4&=\{\xi\in\Xred:
             -1<w_-<-1+\eta\}.
\end{align*}
By Lemma~\ref{lem:return},
the return time is greater than
$\log(1/\eta)$ on $E_1$,
and greater than
$\frac12\log(1/\eta)$
on both $E_2$ and $E_3$.
Thus
\[
\begin{aligned}
 \#(\Xred_\Delta\cap E_1)\le\frac{2hR_K}{\log(1/\eta)},\qquad
 \#(\Xred_\Delta\cap E_i)\le\frac{4hR_K}{\log(1/\eta)}
 \quad (i=2,3).
\end{aligned}
\]
It remains to control $E_4$.
Suppose that $\Phi(\xi)\in E_4$.
Then
\[
 -1<w_-(\Phi(\xi))<-1+\eta.
\]
Since
$w_-(\Phi(\xi))
 =\frac1{w_--n}$ and $n=\lfloor w_+\rfloor$,
we have $n=1$.
Consequently,
\[
 w_-
 =1+\frac1{w_-(\Phi(\xi))}\in
 \left(-\frac{\eta}{1-\eta},0\right)
 \subset(-2\eta,0).
\]
Therefore,
\[
 \Phi^{-1}(E_4)
 \subset\{\xi\in\Xred:
             -2\eta<w_-<0\}.
\]
Since $\Phi$ is a bijection
on the finite set $\Xred_\Delta$,
we obtain
\[
\begin{aligned}
 \#(\Xred_\Delta\cap E_4)
 &=\#\{\xi\in\Xred_\Delta:
       \Phi(\xi)\in E_4\}\\
 &\le
 \#\{\xi\in\Xred_\Delta:
       -2\eta<w_-<0\}\\
 &\le
 \frac{4hR_K}{\log(1/(2\eta))}.
\end{aligned}
\]
The last inequality follows
from Lemma~\ref{lem:return}
and the counting bound above.

Combining the four estimates,
we obtain
\[
 \#(\Xred_\Delta\setminus K_\eta)
 \le2hR_K\,e(\eta),
\]
where
\[
 e(\eta)=
 \frac5{\log(1/\eta)}
 +\frac2{\log(1/(2\eta))}.
\]
In particular,
\[
 e(\eta)\longrightarrow0
 \qquad(\eta\to0^+).
\]

\medskip
\noindent
\emph{Step 3. Conclusion.}
Let $B\subset\Xred$ be a Borel set
whose boundary in $\Xred$ is $\nu$-null,
as in the statement of the proposition.
Put
\[
 B_\eta=B\cap K_\eta.
\]
Since $K_\eta$ is a compact rectangle
contained in $\Xred$, the set $B_\eta$
has compact closure in $\Xred$.
The boundary of $K_\eta$ consists
of four line segments and hence has
two-dimensional Lebesgue measure zero.
Since $\nu$ is absolutely continuous
with respect to Lebesgue measure,
we have $\nu(\partial K_\eta)=0$.
Moreover,
\[
 \partial B_\eta
 \subset \partial B\cup\partial K_\eta.
\]
By the assumption on $B$, it follows
that $\nu(\partial B_\eta)=0$.

By Step 1,
\[
 \frac{\#(\Xred_\Delta\cap B_\eta)}
 {2hR_K}
 \longrightarrow
 \frac{2\nu(B_\eta)}{\Lv(Y)}.
\]
On the other hand, Step 2 gives
\[
 0\le
 \frac{
 \#(\Xred_\Delta\cap
       (B\setminus K_\eta))
 }{2hR_K}
 \le e(\eta).
\]
We also have
\[
 0\le\nu(B\setminus K_\eta)
 \le\nu(\Xred\setminus K_\eta).
\]
Combining these estimates,
we obtain
\[
\limsup_{\Delta\to\infty}
 \left|
 \frac{\#(\Xred_\Delta\cap B)}
 {2hR_K}
 -\frac{2\nu(B)}{\Lv(Y)}
 \right|\le
 e(\eta)
 +\frac{2}{\Lv(Y)}
 \nu(\Xred\setminus K_\eta).
\]
Since
$e(\eta)\to0$ and $\nu(\Xred\setminus K_\eta)\to0$
as $\eta\to0^+$, we conclude that
\[
 \frac{\#(\Xred_\Delta\cap B)}
 {2hR_K}
 \longrightarrow
 \frac{2\nu(B)}{\Lv(Y)}
\]
for every $B$ satisfying
the assumptions of the proposition.
Taking $B=\Xred$ and using
$\nu(\Xred)=\log2$, we obtain
\[
 \frac{\kappa(\Delta)}{2hR_K}
 \longrightarrow
 \frac{2\log2}{\Lv(Y)}.
\]
The hyperbolic area of
$\PSL_2(\Z)\backslash\HH$ is $\pi/3$.
Since $\PSL_2(\Z)$ has index two in
$\Gamma=\PGL_2(\Z)$, the hyperbolic
area of $\Gamma\backslash\HH$ is $\pi/6$.
Our normalization of the Liouville measure
gives angular fibers of volume $2\pi$.
Consequently,
\[
 \Lv(Y)=2\pi\cdot\frac{\pi}{6}
 =\frac{\pi^2}{3}.
\]
It follows that
\[
 \frac{\kappa(\Delta)}{hR_K}
 \longrightarrow
 \frac{4\log2}{\pi^2/3}
 =\frac{12\log2}{\pi^2}.
\]
Finally, dividing the general
limit by the preceding one gives
\[
 \frac{\#(\Xred_\Delta\cap B)}
 {\kappa(\Delta)}
 \longrightarrow
 \frac{\nu(B)}{\log2}.
\]
This completes the proof.
\end{proof}

\begin{proof}[Proof of Theorem~\ref{thm:intro-individual}]
By Lemma~\ref{lem:linear-coordinates},
we have
\[
 \ECal(\Delta)
 =\#(\Xred_\Delta\cap\Xex).
\]
Similarly, the number of reduced ideals
satisfying $\lfloor f\rfloor=1$ is
$\#(\Xred_\Delta\cap\Xone)$.
The boundaries of $\Xex$ and $\Xone$
in $\Xred$ are a curve and a line segment,
respectively.
Both are $\nu$-null.
Thus Proposition~\ref{prop:equidistribution}
and Lemma~\ref{lem:measures} give
$\frac{\ECal(\Delta)}{\kappa(\Delta)}
 \rightarrow\lex$ and $\frac{\#(\Xred_\Delta\cap\Xone)}
 {\kappa(\Delta)}
 \rightarrow
 \frac{\log(4/3)}{\log2}$.
Consequently,
\[
 \frac{\ECal(\Delta)}
 {\#(\Xred_\Delta\cap\Xone)}
 \longrightarrow\frac12.
\]
Proposition~\ref{prop:equidistribution}
gives
$\kappa(\Delta)
 \sim\frac{12\log2}{\pi^2}hR_K$.
Combining this with
$\ECal(\Delta)/\kappa(\Delta)\to\lex>0$,
we obtain
\[
 \ECal(\Delta)
 \sim\lex\frac{12\log2}{\pi^2}hR_K
 =\frac{6\log(4/3)}{\pi^2}hR_K.
\]
The analytic class number formula
for real quadratic fields
\cite[pp.~43--53]{DavenportMNT}
gives
\[
 hR_K=\frac12\sqrt\Delta\,
 L(1,\chi_\Delta).
\]
Combining this with the preceding
asymptotic formula, we obtain
\[
 \ECal(\Delta)
 \sim\frac{3\log(4/3)}{\pi^2}
 \sqrt\Delta\,L(1,\chi_\Delta).
\]
By Siegel's theorem
\cite[pp.~126--131]{DavenportMNT},
for every $\varepsilon>0$ we have
\[
 L(1,\chi_\Delta)
 \gg_\varepsilon\Delta^{-\varepsilon},
\]
with an ineffective implied constant.
Since $3\log(4/3)/\pi^2>0$,
it follows that
\[
\begin{aligned}
 \ECal(\Delta)
 &\gg \sqrt\Delta\,L(1,\chi_\Delta)\\
 &\gg_\varepsilon
 \sqrt\Delta\,\Delta^{-\varepsilon}\\
 &=\Delta^{1/2-\varepsilon}.
\end{aligned}
\]
This completes the proof.
\end{proof}

The following lemma gives an explicit
formula for the tail measure of the
defect distribution.

\begin{lemma}\label{lem:defect-tail}
For $t\ge0$, we put
\[G(t)=\nu\bigl(\{\xi\in\Xex:\dfn(\xi)>t\}\bigr).\]
We set
$t_{\max}=\frac1{\sqrt3}-\frac12$.
For $0<t<t_{\max}$, we define
$u_\pm(t)
 =\frac{2t+3\pm\sqrt{1-12t-12t^2}}
 {2(2t+1)}$, and
$F_t(u)
 =(1+t)\log u-t\log(u-1)
 -\frac12\log(u^2-u+1)$.
Then
\[
 G(t)=
 \begin{cases}
 \displaystyle\frac12\log\frac43,
   & t=0,\\[6pt]
 F_t(u_+(t))-F_t(u_-(t)),
   & 0<t<t_{\max},\\[4pt]
 0,
   & t\ge t_{\max}.
 \end{cases}
\]
Moreover, as $t\to0^+$ we have
\[G(0)-G(t)
 =t\log(1/t)+(1-2\log2)t+o(t).\]
\end{lemma}

\begin{proof}
Write $u=w_+$ and $w=-w_-$,
so that $\xi=(u,-w)$.
Then
$d\nu=\frac{du\,dw}{(u+w)^2}$.
In these coordinates, $\Xex$ consists of all
points $\xi=(u,-w)$ satisfying $1<u<2$ and
$0<w<(2-u)/(2u-1)$.
The defect function is given by
\[\dfn(u,-w)=\frac{(u-1)w}{u+w}.\]

Suppose first that $t>0$.
For fixed $u\in(1,2)$, 
since
$\frac{\partial}{\partial w}
 \dfn(u,-w)
 =\frac{u(u-1)}{(u+w)^2}>0$,
we have that $\dfn(u,-w)$ is strictly increasing
with respect to $w$.
Since
$0<w<\frac{2-u}{2u-1}$,
the supremum of $\dfn(u,-w)$
on this interval is
\[
 h(u):=\frac{(u-1)(2-u)}
 {2(u^2-u+1)}.
\]
Consequently, the vertical slice
satisfying $\dfn(u,-w)>t$ is nonempty
if and only if $h(u)>t$.
This inequality is equivalent to
\[
 (2t+1)u^2-(2t+3)u+2(t+1)<0.
\]
The discriminant of this quadratic
polynomial is
\[
 1-12t-12t^2.
\]
It is positive precisely when
$0<t<t_{\max}$.
In this range, the two roots are
$u_-(t)$ and $u_+(t)$.
Consequently, the slice is nonempty
exactly for
\[
 u_-(t)<u<u_+(t).
\]
For such $u$, the inequality
$h(u)>t$, together with
$h(u)<u-1$, implies $u>1+t$.
Therefore, the condition
$\dfn(u,-w)>t$ is equivalent to
\[
 w>\frac{tu}{u-1-t}.
\]
Thus the required region is
described by
\[
 u_-(t)<u<u_+(t),
 \qquad
 \frac{tu}{u-1-t}
 <w<\frac{2-u}{2u-1}.
\]
Integrating first with respect to
$w$, we obtain
\[
\begin{aligned}
 \int_{tu/(u-1-t)}^{(2-u)/(2u-1)}
 \frac{dw}{(u+w)^2}&=
 \frac{u-1-t}{u(u-1)}
 -\frac{2u-1}{2(u^2-u+1)}\\
 &=\frac{d}{du}F_t(u).
\end{aligned}
\]
Therefore,
\[
\begin{aligned}
 G(t)
 &=\int_{u_-(t)}^{u_+(t)}
 \frac{d}{du}F_t(u)\,du\\
 &=F_t(u_+(t))-F_t(u_-(t)).
\end{aligned}
\]
If $t\ge t_{\max}$, the quadratic
inequality has no solution.
Thus $G(t)=0$.
For $t=0$, Lemma~\ref{lem:measures}
gives
\[
 G(0)=\nu(\Xex)=\frac12\log\frac43.
\]

It remains to establish the
asymptotic expansion as $t\to0^+$.
Since
\[
 \sqrt{1-12t-12t^2}
 =1-6t+O(t^2),
\]
we have
\[
 u_-(t)=1+2t+O(t^2),
 \qquad
 u_+(t)=2-6t+O(t^2).
\]
We write
$F_t(u)=F_0(u)+t\log\frac{u}{u-1}$,
where
\[
 F_0(u)=\log u
 -\frac12\log(u^2-u+1).
\]
Differentiating gives
$F_0'(u)
 =\frac{2-u}{2u(u^2-u+1)}$.
In particular,
$F_0'(1)=\frac12$ and $F_0'(2)=0$.
Taylor expansions at $u=1$ and
$u=2$ give
\[
 F_0(u_-(t))=F_0(1)+t+O(t^2),\qquad
 F_0(u_+(t))=F_0(2)+O(t^2).
\]
Since
$F_0(2)-F_0(1)=G(0)$,
we obtain
\[
 F_0(u_+(t))-F_0(u_-(t))
 =G(0)-t+O(t^2).
\]
On the other hand, the expansions
of $u_-(t)$ and $u_+(t)$ imply
\[
 \log\frac{u_+(t)}{u_+(t)-1}=\log2+O(t),\qquad
 \log\frac{u_-(t)}{u_-(t)-1}=-\log(2t)+O(t).
\]
Substituting these estimates into
\[
 G(t)
 =F_0(u_+(t))-F_0(u_-(t))
 +t\log\frac{u_+(t)}{u_+(t)-1}
 -t\log\frac{u_-(t)}{u_-(t)-1},
\]
we obtain
\[
\begin{aligned}
 G(t)
 &=G(0)-t
   +t\log2+t\log(2t)+O(t^2)\\
 &=G(0)-t+t\log(4t)+O(t^2).
\end{aligned}
\]
Hence
$G(0)-G(t)
 =t\log(1/t)
 +(1-2\log2)t+O(t^2)$,
which completes the proof.
\end{proof}

\begin{proof}[Proof of Theorem~\ref{thm:intro-defect}]
For $s\ge0$, put
\[
 B_s=\{\xi\in\Xex:\dfn(\xi)>s\}.
\]
Let $N_\Delta(t)$ denote the number of
reduced ideals $I$ that are not $1$-reduced
and satisfy $k(I)>t\sqrt\Delta$.
By Lemma~\ref{lem:linear-coordinates}(iv),
every reduced ideal $I$ that is not
$1$-reduced satisfies
\[
 \dfn(\xi_I)-\Delta^{-1/2}
 <\frac{k(I)}{\sqrt\Delta}
 <\dfn(\xi_I).
\]
Fix $t\ge0$.
If $\xi_I\in B_{t+\Delta^{-1/2}}$,
then
$\dfn(\xi_I)>t+\Delta^{-1/2}$.
Therefore,
\[
 \frac{k(I)}{\sqrt\Delta}
 >\dfn(\xi_I)-\Delta^{-1/2}>t.
\]
This implies that $k(I)>t\sqrt\Delta$,
and hence $I$ contributes to $N_\Delta(t)$.
Conversely, if $I$ is counted
by $N_\Delta(t)$, then
\[
 \dfn(\xi_I)
 >\frac{k(I)}{\sqrt\Delta}>t.
\]
Thus $\xi_I\in B_t$.
Since the points of
$\Xred_\Delta\cap\Xex$ correspond
to the reduced ideals that are not
$1$-reduced, we obtain
\[
 \#(\Xred_\Delta\cap B_{t+\Delta^{-1/2}})
 \le N_\Delta(t)
 \le \#(\Xred_\Delta\cap B_t).
\]

We next verify that
Proposition~\ref{prop:equidistribution}
applies to $B_s$ for every fixed $s\ge0$.
The boundary of $\Xex$ in $\Xred$
is a smooth curve and therefore has
two-dimensional Lebesgue measure zero.
Moreover, writing $u=w_+$ and
$w=-w_-$, we have
\[
 \frac{\partial}{\partial w}
 \dfn(u,-w)
 =\frac{u(u-1)}{(u+w)^2}>0.\]
Thus, for each fixed $u>1$, the
equation $\dfn(u,-w)=s$ has
at most one solution in $w$.
By Fubini's theorem, the level set
\[
 \{\xi\in\Xred:\dfn(\xi)=s\}
\]
has two-dimensional Lebesgue
measure zero.
Since $\nu$ is absolutely continuous
with respect to Lebesgue measure,
both sets are $\nu$-null.
Furthermore, the continuity of
$\dfn$ implies
\[
 \partial B_s
 \subset
 \partial\Xex
 \cup
 \{\xi\in\Xred:\dfn(\xi)=s\},
\]
where the boundaries are taken
in $\Xred$.
Consequently,
\[
 \nu(\partial B_s)=0.
\]
By Proposition~\ref{prop:equidistribution},
we obtain
$\frac{\#(\Xred_\Delta\cap B_s)}
 {\kappa(\Delta)}
 \rightarrow
 \frac{\nu(B_s)}{\log2}
 =\frac{G(s)}{\log2}$.
Theorem~\ref{thm:intro-individual}
gives
$\frac{\ECal(\Delta)}{\kappa(\Delta)}
\rightarrow
 \frac{G(0)}{\log2}>0$.
Combining these two limits, we obtain
\[
 \frac{\#(\Xred_\Delta\cap B_s)}
 {\ECal(\Delta)}
 \longrightarrow
 \frac{G(s)}{G(0)}.
\]
Now fix $t\ge0$ and $\varepsilon>0$.
For sufficiently large $\Delta$,
we have $\Delta^{-1/2}<\varepsilon$.
Since
$B_{t+\varepsilon}
\subset B_{t+\Delta^{-1/2}}$,
we have
\[
\begin{aligned}
 \#(\Xred_\Delta\cap B_{t+\varepsilon})
 &\le
 \#(\Xred_\Delta\cap B_{t+\Delta^{-1/2}})\\
 &\le N_\Delta(t)\\
 &\le \#(\Xred_\Delta\cap B_t).
\end{aligned}
\]
Dividing by $\ECal(\Delta)$
and letting $\Delta\to\infty$,
we obtain
\begin{equation}\label{eq:defect-squeeze}
\begin{aligned}
 \frac{G(t+\varepsilon)}{G(0)}
 &\le
 \liminf_{\Delta\to\infty}
 \frac{N_\Delta(t)}{\ECal(\Delta)}\\
 &\le
 \limsup_{\Delta\to\infty}
 \frac{N_\Delta(t)}{\ECal(\Delta)}
 \le\frac{G(t)}{G(0)}.
\end{aligned}
\end{equation}
Let $s_n\ge0$ and suppose that
$s_n\to s\ge0$.
Fix $\xi\in\Xex$ satisfying
$\dfn(\xi)\ne s$.
Since $s_n\to s$, the inequalities
$\dfn(\xi)>s_n$ and $\dfn(\xi)>s$
have the same truth value for all
sufficiently large $n$.
Hence the indicator functions
$\mathbf 1_{\{\dfn>s_n\}}$
converge pointwise to
$\mathbf 1_{\{\dfn>s\}}$
on $\Xex\setminus\{\dfn=s\}$.
Since the level set
$\{\xi\in\Xex:\dfn(\xi)=s\}$
is $\nu$-null, this convergence
holds $\nu$-almost everywhere
on $\Xex$.
Moreover, these indicator functions
are bounded by the constant function
$1$, which is $\nu$-integrable on
$\Xex$ since $\nu(\Xex)<\infty$.
Since
\[
 G(s_n)
 =\int_{\Xex}
 \mathbf 1_{\{\dfn>s_n\}}\,d\nu,
\]
the dominated convergence theorem gives
\[
\begin{aligned}
 \lim_{n\to\infty}G(s_n)
 &=\int_{\Xex}
   \mathbf 1_{\{\dfn>s\}}\,d\nu\\
 &=G(s).
\end{aligned}
\]
Thus $G$ is continuous on $[0,\infty)$.
Since $G$ is continuous, letting
$\varepsilon\to0^+$ in
\eqref{eq:defect-squeeze} shows that
the lower and upper limits coincide.
Thus
\[
 \frac{N_\Delta(t)}{\ECal(\Delta)}
 \longrightarrow
 \frac{G(t)}{G(0)}.
\]
By the definition of $N_\Delta(t)$,
we have
\[
\begin{aligned}
 &\frac{
 \#\{I\text{ reduced, not $1$-reduced}:
 k(I)\le t\sqrt\Delta\}
 }{\ECal(\Delta)}=
 1-\frac{N_\Delta(t)}{\ECal(\Delta)}
 \longrightarrow
 1-\frac{G(t)}{G(0)}.
\end{aligned}
\]
Thus the limiting distribution function
is $1-G(t)/G(0)$ for $t\ge0$.
By Lemma~\ref{lem:defect-tail},
we have $G(t)=0$ for $t\ge t_{\max}$.
Since the normalized defect is
nonnegative, the limiting distribution
is supported on $[0,t_{\max}]$.

Finally, Lemma~\ref{lem:defect-tail}
gives
$G(0)-G(t)
 =t\log(1/t)+(1-2\log2)t+o(t)$
as $t\to0^+$.
Using
$G(0)=\frac12\log\frac43$,
we obtain
\[
\begin{aligned}
 1-\frac{G(t)}{G(0)}
 &=\frac{G(0)-G(t)}{G(0)}\\
 &=\frac{2}{\log(4/3)}
 \left(
 t\log\frac1t
 +(1-2\log2)t
 \right)+o(t).
\end{aligned}
\]
This completes the proof.

\end{proof}

\section*{Declaration of competing interests}

The authors declare no competing interests.

\section*{Funding sources}

This research did not receive any specific grant
from funding agencies in the public, commercial,
or not-for-profit sectors.

\end{document}